\documentclass[11pt]{amsart}
\usepackage{graphicx}
\usepackage{latexsym,amsmath,amsfonts,amscd, amsthm}
\usepackage{color}
\usepackage[colorlinks=true]{hyperref}
\hypersetup{urlcolor=blue, citecolor=red}
\usepackage{tikz}
\usetikzlibrary{arrows,backgrounds,snakes}
\usepackage{epsfig}
\usepackage{amssymb}
\usepackage{xcolor}
\usepackage{comment}
\usepackage{subfigure}
\usepackage{multirow}
\newtheoremstyle{plainNoItalics}{}{}{\normalfont}{}{\bfseries}{.}{ }{}
\theoremstyle{plain}
\newtheorem{thm}{Theorem}[section]

\newtheorem{defn}[thm]{Definition}
\newtheorem{rem}[thm]{Remark}
\newtheorem{prop}[thm]{Proposition}

\newcommand{\f}{\frac}

\newcommand{\beq}{\begin{equation}}
\newcommand{\eeq}{\end{equation}}
\newcommand{\beqa}{\begin{eqnarray}}
\newcommand{\eeqa}{\end{eqnarray}}
\newcommand{\bit}{\begin{itemize}}
\newcommand{\eit}{\end{itemize}}
\newcommand{\bedef}{\begin{defn}}
\newcommand{\edefn}{\end{defn}}
\newcommand{\bpro}{\begin{prop}}
\newcommand{\epro}{\end{prop}}
\newcommand{\RR}{\mathbb{R}}
\newcommand{\NN}{\mathbb{N}}

\newcommand\bx{{\mathbf x}}
\newcommand\bv{{\mathbf v}}
\newcommand\bE{{\mathbf E}}
\DeclareMathOperator\dD{d}
\newcommand{\mE}{\mathcal E}
\newcommand{\mH}{\mathcal H}

\newcommand{\bC}{{\mathbf C}}

\newcommand{\xL}{{x_{i-\frac{1}{2}}}}
\newcommand{\xR}{{x_{i+\frac{1}{2}}}}

\newcommand{\iR}{{i+\frac{1}{2}}}

\newcommand{\testR}{{\phi}}   % test function
\newcommand{\testP}{{\eta}}   % test function
\newcommand{\testE}{{\zeta}}

\newcommand{\pf}{\partial f}
\newcommand{\pt}{\partial t}
\newcommand{\px}{\partial x}

\newcommand{\pv}{\partial v}

\newcommand\ds{ \displaystyle }

\email{francis.filbet@math.univ-toulouse.fr}
\email{txiong@xmu.edu.cn}

\title[Conserative method for the Vlasov-Poisson system]{Conservative discontinuous Galerkin/Hermite Spectral Method for the  Vlasov-Poisson System} 

\author{Francis Filbet  and Tao Xiong}



\begin{document}

\maketitle

\centerline{\scshape Francis Filbet}
\medskip
{\footnotesize
    % please put the address of the second  and third author
    \centerline{Institut de Math\'ematiques de Toulouse , Universit\'e Paul Sabatier}
    \centerline{Toulouse, France}
}

\medskip

\centerline{\scshape Tao Xiong}
\medskip
{\footnotesize
    % please put the address of the second  and third author
    \centerline{School of Mathematical Sciences, Xiamen University}
    \centerline{Fujian Provincial Key Laboratory of Math. Mod. \& HPSC}
    \centerline{Xiamen,  Fujian, 361005, P.R. China}
  }
  
\bigskip

\begin{abstract}
We propose a class of conservative discontinuous Galerkin methods  for
the Vlasov-Poisson system written as a hyperbolic system using
Hermite polynomials in the velocity variable. These schemes are  designed to be
systematically as accurate as one wants with provable conservation of
mass and possibly  total energy.  Such properties in general are hard
to achieve within other numerical method frameworks for simulating the
Vlasov-Poisson system. The proposed scheme employs discontinuous
Galerkin discretization for both the Vlasov and the Poisson
equations, resulting in a consistent description of the distribution
function and electric field.  Numerical simulations are performed to
verify the  order of accuracy  and  conservation properties.   
\end{abstract}

\vspace{0.1cm}

\tableofcontents

\section{Introduction}
\label{sec1}
\setcounter{equation}{0}
\setcounter{figure}{0}
\setcounter{table}{0}
One of the simplest model  that is currently applied in plasma physics
simulations is the Vlasov-Poisson system. Such a system
describes the evolution  of charged particles (electrons
and ions) under the effects of a self-consistent
electric field. For each species labelled $\alpha$, 
the unknown $f_\alpha(t,\bx,\bv)$, depending on the time $t$, the position $\bx$, and the velocity
$\bv$, represents the distribution function of particles in phase space. This model can be used for the study of beam propagation or 
a collisionless plasma
\begin{equation}
  \label{vlasov0}
  \left\{
    \begin{array}{l}
\ds\frac{\partial f_\alpha}{\partial t}\,+\,\bv\cdot\nabla_\bx f_\alpha \,+\,\frac{q_\alpha}{m_\alpha}\bE\cdot\nabla_\bv f_\alpha \,=\, 0\,,
      \\[0.9em]
      f_\alpha(t=0) = f_{\alpha,0}\,,
      \end{array}\right.
      \end{equation}
where $q_\alpha$ represents  the single charge and $m_\alpha$
represents the mass of one particle $\alpha$, whereas the electric
field $\bE= -\nabla_\bx \Phi$ satisfies the Poisson equation
\begin{equation}
\label{poisson0}
-4\pi\,\epsilon_0\,\Delta_\bx \Phi \,=\, \sum_{\alpha} q_\alpha
n_\alpha\,,\qquad{\rm with }\quad n_\alpha = \int_{\RR^3} f_\alpha \,\dD\bv\,,
\end{equation}
where $\epsilon_0$ is the vacuum permittivity. On the one hand, for a smooth and nonnegative initial data $f_{\alpha,0}$, the solution $f_\alpha(t)$
to \eqref{vlasov0} remains smooth and nonnegative for all $t\geq
0$.  On the other hand, for any function  $\beta\in C^1(\RR^+,\RR^+)$, we have
$$
\frac{\dD}{\dD t}\int_{\RR^d \times \RR^d} \beta(f_\alpha(t)) \,\dD \bx \dD\bv
=0,\quad\forall t\in\RR^+\,,
$$
which leads to the conservation of mass,  $L^p$ norms, for $1 \leq p \leq
+\infty$ and  kinetic entropy, 
$$
\mH(t) \,:=\, \int_{\RR^6}
f_\alpha(t)\,\ln\left(f(t)\right)\dD\bx \dD\bv \,=\, \mH(0), \quad\forall
t\geq 0\,.
$$
 We also get the conservation of momentum
$$
\sum_{\alpha} \int_{ \RR^6} m_\alpha \, \bv \,
f_\alpha(t)\,\dD\bx\dD\bv =  \sum_{\alpha} \int_{\RR^6} m_\alpha \, \bv \,
f_{\alpha,0}\,\dD\bx\dD\bv 
$$
and total energy
$$
\mE(t)\, :=\, \sum_{\alpha}\frac{m_\alpha}{2}\int_{\RR^6}
f_\alpha(t) \|\bv\|^2 \dD\bx \dD\bv \,+\, 2\pi\epsilon_0 \int_{\RR^3}
\|\bE\|^2 \dD\bx   \,=\, \mE(0)\,, \quad\forall
t\geq 0\,.
$$

The numerical resolution of the Vlasov-Poisson system \eqref{vlasov0}-\eqref{poisson0} is often performed
by particle methods (PIC) which consist in approximating the plasma by
a finite number of macro-particles. The trajectories of these particles are
computed from the characteristic curves given by the Vlasov equation,
whereas self-consistent fields are computed on a mesh of the physical
space. This method allows to obtain satisfying results with a few
number of particles (we refer to Birdsall-Langdon for more details
\cite{birdsall}). However, it suffers from  poor accuracy since  the numerical noise only decreases in $1/\sqrt N$ when
the number of particles $N$ is increased, which is not satisfying for
some specific problems. Therefore, different approaches, discretizing the Vlasov equation on
a mesh of phase space, have been proposed. Among them, the
Fourier-Fourier transform is based on a Fast Fourier Transform of the
distribution function in phase space, but this method is only
valid for periodic boundary conditions
\cite{klimas1}. Consequently, for non periodic boundary
conditions, Gibbs oscillations form at the boundary of the grid and
become source of spurious oscillations which propagate into the
distribution function.  Later semi-Lagrangian methods, which consist in computing the distribution
function at each grid point by following the characteristic curves backward, were also used. To compute the
origin of the characteristics, a high order interpolation method is
needed.  E. Sonnendr\"ucker $et$ $al.$ proposed the cubic spline reconstruction which gives very good results \cite{sonnen1}, but the
use of spline interpolation destroys the local 
character of the reconstruction. Nakamura and Yabe also presented  the Cubic
Interpolated Propagation (CIP) method based on the approximation of
the gradients of the distribution function in order to use a Hermite
interpolation \cite{yabe}. This method is very expensive in memory computation
since it needs the storage of $f$, $\nabla_xf$, and $\nabla_vf$.
Another scheme for the Vlasov equation is the Flux Corrected Transport
(FCT) \cite{boris} or more recently  finite volume methods \cite{filbet2001, filbet2003,FDD,oudet}: the basic idea is to compute the average of the
Vlasov equation solution in each cell of the phase space grid by a
conservative method. In the sixties, a variant has been proposed: rather to discretize the
function in velocity space, Galerkin methods with a “small” finite set
of orthogonal polynomials are used \cite{armstrong1967,
  joyce1971}. More recently, in \cite{tang, holloway} and \cite{holloway2},
the authors have shown the merit to use rescaled orthogonal basis like
the so-called scaled Hermite basis. In \cite{holloway2}, the authors
show that, according to some symmetry properties of the Hermite
weights, the numerical method conserves certain integral quantities
like the number of particles, the momentum, the energy or the
$L^2$-norm. These properties lead to algorithms that can provide long
time numerical stability while using a small set of basis
functions. The aim of the present work is to apply these techniques in order
to handle with the velocity space and design conservative
discretization with a reduced degree of freedom.

 For the space discretization, we adopt the point of view of
 discontinuous Galerkin  methods.  These methods are similar to finite
 element methods but use discontinuous polynomials and are
 particularly well suited to handle complicated 
boundaries which may arise in many practical applications. Moreover,  their local
 construction endows the methods with good local conservation
 properties without sacrificing the order of accuracy.  Furthermore, they
 are extremely flexible in handling $hp$-adaptivity,  the boundary
 conditions are imposed weakly and the mass matrices are
 block-diagonal which results in very efficient time-stepping
 algorithms in the context of time-dependent problems, as it is the
 case here. Among the computational works, we mention the use of
 discontinuous Galerkin schemes for the Vlasov-Poisson system in
 \cite{yingda-2013,heath-2012};  and Vlasov-Maxwell
 \cite{yingda-2014}. Several theoretical works have been performed to analyze a
 family of semi-discrete discontinuous Galerkin schemes for the
 Vlasov-Poisson system  with periodic boundary conditions, for the one
 \cite{ayuso2009} and multi-dimensional cases \cite{ayuso2012} , but
 also for the Vlasov-Maxwell system \cite{fengyan-2015}. The authors
 show optimal error estimates for both the distribution function and
 the electrostatic field, and they study the conservation properties
 of the proposed schemes. Let us emphasize that these former works
 apply a discontinuous Galerkin method using a phase space mesh,
 whereas here we adopt this approach only in physical
 space. Furthermore,  we propose to modify the fluxes to ensure conservation of mass,
 momentum and total energy. We prove these conservations for the
 semi-discrete and fully discrete cases (first and second order time discretizations).

This paper is organized as follows: in the first part (Section \ref{sec:2}), we briefly
reformulate the Vlasov equation using the Hermite basis in velocity, recalling some properties of the
solution. Then, in Section \ref{sec:3}, we present a variant of the discontinuous
Galerkin method for the space discretization such that mass, momentum
and energy are preserved for semi-discrete and fully discrete
systems. In the last section (Section \ref{sec:4}), we present numerical results for Landau
damping, two stream instability and bump on the tail problems to
illustrate accuracy and conservations of our discretization   technique.

%%%%%%%%%%%%%%%%%%%%%%%%%%%%%%%%%%%%%%%%%%
%
%%%%%%%%%%%%%%%%%%%%%%%%%%%%%%%%%%%%%%%%%%
\section{Hermite spectral form of the Vlasov equation}
\label{sec:2}
\setcounter{equation}{0}
\setcounter{figure}{0}
\setcounter{table}{0}

For simplicity, we now set all the physical constants to one and consider the one dimensional Vlasov-Poisson
system for a single species with periodic boundary conditions in space, where the Vlasov equation reads
\beq
\label{vlasov}
\f{\pf}{\pt}+v\f{\pf}{\px} +E\f{\pf}{\pv} = 0\,
\eeq
with $t\geq 0$, position $x \in (0,L)$ and velocity
$v\in\RR$. Moreover, the self-consistent electric field $E$  is determined by the Poisson equation
\beq
\label{poisson}
\f{\partial E}{\px}=\rho- \rho_0\,,
\eeq
where the density $\rho$ is given by
$$
\rho(t,x)\,=\,\int_\RR f(t,x,v)\,\dD v\,, \quad t\geq 0, \, x\in (0,L)
$$
and $\rho_0$ is a constant  ensuring the quasi-neutrality condition
of the plasma
\begin{equation}
\int_{0}^L \left(\rho - \rho_0\right) \,\dD x = 0\,. 
\label{back:0}
\end{equation}

\subsection{Hermite basis}
To discretize the velocity variable, our approach  is based on spectral methods, where we
expand  the velocity part of the distribution function in basis
functions (typically Fourier or Hermite), leading to a truncated set
of moment equations for the expansion coefficients. Actually, spectral
methods are commonly used to approximate the solution to partial differential
equations \cite{boyd2001,hesthaven2007} and in particular the
Vlasov-Poisson system \cite{shoucri1974,klimas1,eliasson2003, holloway,holloway2}.

For instance, in their seminal work \cite{shoucri1974},  M. Shoucri and G. Knorr applied Chebyshev polynomials, whereas in
\cite{klimas1,eliasson2003} the authors used Fourier basis but these
methods does not conserve neither momentum, nor total energy.  Later,
J. P. Holloway  and J. W.  Schumer \cite{holloway,holloway2} applied
Hermite functions. Indeed, the product of Hermite
polynomials and a Gaussian, seems to be a natural choice for
Maxwellian-type velocity profiles \cite{funaro1}. Moreover,  compared with other
classical polynomials, they possess a fairly simple expression for
their derivatives and allow to get conservation of mass, momentum and
total energy. Recently, these methods generate a new interest
leading to new techniques to improve their efficiency \cite{le2006,
  cai2013,camporeale2016, manzini2016, cai2018, filtered} .

Let us summarize the main characteristics of this approach. The solution $f$ is approximated by a finite sum which corresponds to a truncation of a
series
\beq
\label{fseries}
f(t,x,v)=\sum_{n=0}^{N_H-1}C_n(t,x)\Psi_n(v)\,,
\eeq
where $N_H$ is the number of modes. The issue is then to
determine a well-suited class of basis functions $\Psi_n$ and to find
the expansion coefficients $C_n$. Here, we have chosen the following basis of normalized scaled {asymmetrically weighted} Hermite functions:
\beq
\label{hbasisf}
\Psi_n(v)\,=\,{H}_n\left(\frac{v}{v_{th}}\right)\,\f{e^{-\ds{v^2}/{2 v_{th}^2}}}{\sqrt{2\pi}}\,,
\eeq
where $v_{th}$ corresponds to the scaling velocity and  we have set  ${H}_{-1}(\xi)=0$, ${H}_0(\xi)=1$ and for $n\geq
1$, ${H}_n(\xi)$ has the following recursive relation 
\beq
\sqrt{n}\,{H}_n(\xi) \,=\, \xi \,{H}_{n-1}(\xi)-\sqrt{n-1}\,{H}_{n-2}(\xi)\,, \quad \forall \,n \ge 1\,.
\eeq
The Hermite basis \eqref{hbasisf}  has the following properties
\beq
\label{h1}
\int_{\RR} \Psi_n(\xi)\,\Psi_m(\xi) \,{\frac{e^{\xi^2/2}}{\sqrt{2\pi}}}\,\dD\xi \,=\, \delta_{n,m}\,, 
\eeq
and $\delta_{n,m}$ is the Kronecker delta function. With those properties,
one can substitute the expression \eqref{fseries} for $f$ into the Vlasov
equation \eqref{vlasov} and using the orthogonality property
\eqref{h1}, it yields an evolution equation for each mode $C_n$, 
\beq
\label{cn}
\f{\partial C_n}{\pt}\,+\,v_{th}\left(\sqrt{n+1}\,\f{\partial C_{n+1}}{\px}\,+\,\sqrt{n}\,\f{\partial C_{n-1}}{\px} \right)-\f{\sqrt{n}}{v_{th}}\,E\,C_{n-1} \,=\, 0\,.
\eeq
Meanwhile, we first observe that the density $\rho$ satisfies
$$
\rho \,=\, \int_\RR f \dD v \,=\, v_{th}\, C_0\,,
$$
then the Poisson equation becomes 
\beq
\label{ps}
\f{\partial E}{\px} = v_{th}\,C_0 - \rho_0\,.
\eeq
Observe that if we take $N_H=\infty$ in the expression \eqref{fseries},
we get an infinite system \eqref{cn}-\eqref{ps} of equations for $(C_n)_{n\in\NN}$ and $E$, which is
formally equivalent to the Vlasov-Poisson  system
\eqref{vlasov}-\eqref{poisson}.

\subsection{Conservation properties}

There are different kinds of Hermite approximations based
on the choice of Hermite functions and weights.  In
\cite{holloway,holloway2},  J. W. Schumer and J. P. Holloway have
discussed precisely the different choices of orthogonal
Hermite basis functions depending on the form of their weight
functions. If the basis is {asymmetrically weighted}, hence mass, momentum and total energy are preserved
but not the $L^2$ norm. {In the case of symmetrically weighted basis
functions, particles number, total energy (for $N_H$ odd) or momentum (for $N_H$ even),  as well as
$L^2$-norm are conserved} and the numerical stability is shown to be
better. Moreover, it has been shown in \cite{tang} and later in \cite{le2006}  that it is required to introduce a velocity
scaling factor to make the method more accurate and stable.  Here,
the basis is {asymmetrically weighted}, hence we prove the following result for the truncated system
\eqref{cn}-\eqref{ps}.

\begin{prop}
  \label{prop:1}
  For any $N_H\in \NN$, consider the distribution function $f$ given
  by the truncated series
  $$
f(t,x,v)=\sum_{n=0}^{N_H-1}C_n(t,x)\Psi_n(v)\,,
$$
where {$(C_n,E)$} is a solution to the Vlasov-Poisson system written
as \eqref{cn}-\eqref{ps}. Then
mass, momentum and total energy are conserved, that is, 
$$
v_{th}^{k+1}\,\frac{\dD }{\dD t}\int_0^L C_k \, \dD x \, =\, 0, \quad k=0, \,1
$$
and for the total energy,   
$$
\mE(t) \,:=\, \frac{1}{2}\,\int_{0}^L \, v_{th}^3 {\left(\sqrt{2} \,C_2
 + C_0 \right)} \,+\, |E|^2 \dD x   \,=\, \mE(0)\,.
$$
\end{prop}
\begin{proof}
We consider  the first three equations of \eqref{cn} given by 
\begin{equation}
  \label{c0}
  \left\{
    \begin{array}{l}
	\ds\f{\partial C_0}{\pt}\,+\,v_{th}\,\f{\partial C_1}{\px} \,=\, 0\,, \\[0.9em]
	\ds\f{\partial C_1}{\pt}\,+\,v_{th}\left(\sqrt{2} \f{\partial C_2}{\px} + \f{\partial C_0}{\px}\right) \,-\, \frac{1}{v_{th}}\,E\,C_0  \,=\, 0\,,  \\[0.9em]
	\ds\f{\partial C_2}{\pt}\,+\,v_{th}\left(\sqrt{3} \f{\partial C_3}{\px} + \sqrt{2}\f{\partial C_1}{\px}\right) - \frac{\sqrt{2}}{v_{th}} \,E\,C_1 \,=\, 0\,,
    \end{array}
  \right.
  \end{equation}
which will be related to the conservation of mass, momentum and
energy. Thus, let us take a look at the conservation properties from
\eqref{c0}, by considering periodic or compact support boundary
conditions.

First the total mass is defined as
$$
\int_0^L\int_{-\infty}^{\infty}f(t,x,v)\,\dD v\,\dD x \,=\, v_{th}\,\int_0^L C_0(t,x)\dD
x\,,
$$
hence the conservation of mass directly comes from \eqref{c0} by
integrating the first equation with respect to $x\in (0,L)$.

Then we define the  momentum as
$$
\int_0^L\int_{-\infty}^{\infty}v\,f(t,x,v)\,\dD v\,\dD x = v_{th}^2\,\int_0^L
C_1(t,x)\,\dD x\,,
$$
due to the second equation \eqref{c0}  and using the Poisson equation \eqref{ps}, we get
$$
v_{th}^2\,\f{\dD }{\dD t}  \int_0^L C_1\,\dD
x\,=\, v_{th}\,\int_0^L E\,C_0\,\dD x\,=\,\int_0^LE\left(\f{\partial
    E}{\px}+\rho_0\right)\dD x \,=\,\rho_0\int_0^L E\,\dD x\,=\,0\,,
$$
so that the conservation of momentum is obtained.

Finally to prove the conservation of total energy $W$, we compute the
variation of the kinetic energy defined as
$$
\f12\int_0^L\int_{-\infty}^{\infty}v^2 f(t,x,v)\,\dD x\,\dD
v\,=\,{\f{v_{th}^3}{2}\int_0^L\left(\sqrt{2}\,C_2\,+\,C_0\right)\,\dD x}\,.
$$
Thus, combining the first and third equations in \eqref{c0} and
integrating over $x\in (0,L)$, we get 
\beq
\label{bruges}
{\frac{v_{th}^3}{2}\,\f{\dD }{\dD t}  \int_0^L \left(\sqrt{2}\,C_2\,+\,C_0\right)\,\dD
x\,=\, v_{th}^2\,\int_0^L E\, C_1\,\dD x\,.} 
\eeq
On the other hand, multiplying the first equation in \eqref{c0} by $v_{th}\Phi$
and integrating over $x\in (0,L)$, it yields
$$
v_{th}\int_0^L \f{\partial C_0}{\partial t} \Phi \,\dD x \,=\,
- v_{th}^2\,\int_0^L \f{\partial C_1}{\partial x} \,\Phi \,\dD x \,=\,
-v_{th}^2\,\int_0^L C_1\, E \,\dD x\,.  
$$
Using the  Poisson equation \eqref{ps}, we have $v_{th}\,C_0 =
-\partial_{xx}\Phi  + \rho_0$, hence
$$
\f{1}{2}\f{\dD }{\dD t}  \int_0^L |E|^2 \,\dD x \,=\, \int_0^L
\f{\partial}{\partial t}\left(-\f{\partial^2\Phi}{\partial x^2}\right)\, \Phi \,\dD x \,=\, -v_{th}^2\,\int_0^L C_1\, E \,\dD x\,. 
$$
Combining this latter equality with \eqref{bruges}, it yields  the
conservation of total energy 
$$
\f{1}{2}\f{\dD }{\dD t}  \int_0^L \,v_{th}^3\,{\left(\sqrt{2} \,C_2 \,+\,
  C_0\right)} \,+\, |E|^2 \,\dD x  \,=\, 0\,. 
$$
\end{proof}

\begin{rem}
  \label{rem:2.1}
  The natural space associated to our asymmetrically weighted  basis is
  $$
L^2_\omega \,=\, \left\{\, u:\RR\mapsto \RR : \quad \int_\RR
  |u(v)|^2\,e^{|v|^2/2v_{th}^2}\,\dD v \,< \infty \,\right\}. 
$$
Unfortunately there is no estimate of the associated  norm for the solution to
the Vlasov-Poisson system \eqref{vlasov}-\eqref{poisson}, hence there
is no warranty  to get such estimate for the system
\eqref{cn}-\eqref{ps}. It is worth to mention that the symmetrically
weighted basis with the $L^2(\RR)$ norm would be a good choice, but it
would affect the efficiency of the conservation algorithm which will
proposed in the Section \ref{sec:3} (see Remark \ref{rem:3.1}). 
\end{rem}
\subsection{Filtering technique}
Filtering is a common procedure to reduce the effects of the Gibbs
 phenomenon inherent to spectral methods \cite{hesthaven2007}. The
 filter will  consist in multiplying some spectral coefficients in \eqref{fseries}
 by a scaling factor $\sigma$ in order to reduce the amplitude of high
 frequencies, for any $N_H\geq 4$,
 $$
 \widetilde{C}_n\,=\, C_n \,  \sigma\left(\f{n}{N_H}\right).
 $$
 Here, we simply apply a filter, called Hou-Li's filter, already proposed in
 \cite{hou2007computing} for Fourier spectral method, which reads 
 $$
 \sigma(s)\,=\,
 \left\{
   \begin{array}{ll}
     1\,, & \textrm{if\,\,} 0\leq |s|\leq 2/3\,,
     \\[1.1em]
     \exp(-\beta\,|s|^\beta)\,,&\textrm{if \,\,} |s|>2/3\,,
     \end{array}
   \right.
 $$
 where the coefficient $\beta$ is chosen as $\beta = 36$. Actually
 this filter achieved great success in Fourier \cite{hou2007computing} and Hermite  \cite{filtered}
spectral methods. We again refer to \cite{filtered} for a detailled
discussion on the application to the Vlasov-Poisson system using a Hermite spectral method.

\begin{rem}
Observe that the filter is applied only when $N_H\geq 4$, hence the
filtering process does not modify the coefficients $(C_k)_{0\leq k\leq
2}$, that is, conservations of mass, momentum and total energy are not
affected.   
\end{rem}
%%%%%%%%%%%%%%%%%%%%%%%%%%%%%%%%%%%%%%%%%% 
%
%%%%%%%%%%%%%%%%%%%%%%%%%%%%%%%%%%%%%%%%%%
\section{Discontinuous Galerkin scheme}
\label{sec:3}
\setcounter{equation}{0}
\setcounter{figure}{0}
\setcounter{table}{0}
 This section is devoted to the space and time discretizations of the
 Vlasov-Poisson system written in the form \eqref{cn}-\eqref{ps}. We
 apply a local discontinuous Galerkin method for space and introduce a
 slight modification of the fluxes to preserve mass, momentum and
 local energy for this semi-discrete system. Then, we focus on the
 time discretization and prove the conservation for the fully-discrete case. 

\subsection{Semi-discrete discontinuous Galerkin scheme}
\label{sec:3.1}
In this section, we will define the discontinuous Galerkin (DG) scheme
for the Vlasov-Poisson system with Hermite spectral basis in velocity \eqref{cn}-\eqref{ps}. 

Let us first introduce some notations and start with
$\{\xR\}_{i=0}^{i=N_x}$, a partition of $\Omega=[x_\textrm{min},
x_\textrm{max}]$. Here $x_{\frac12}=x_\textrm{min}$,
$x_{N_x+\frac12}=x_\textrm{max}$. Each element is denoted as $I_i=[\xL,
\xR]$ with its length $h_i$, and $h=\max_i h_i$.

Given any non-negative integer $k$, we define a finite dimensional discrete piece-wise polynomial space
\begin{equation*}
V_h^k=\left\{u\in L^2(\Omega): u|_{I_i}\in P_k(I_i), \forall i\right\}\,,
%\label{eq:DiscreteSpace:1mesh}
\end{equation*}
where the local space $P_k(I)$ consists of polynomials of degree at
most $k$ on the interval $I$. We further denote the jump $[u]_\iR$
and the average $\{u\}_\iR$ of $u$ at $x_\iR$ defined as,
$$
[u]_\iR\,=\,{u(x_\iR^+)\,-\,u(x_\iR^-)} \quad{\rm and}\quad
\{u\}_\iR\,=\,\frac12\,\left(u(x_\iR^+)\,+\,u(x_\iR^-)\right)\,, \quad \forall\, i\,,
$$
where $u(x^\pm)=\lim_{\Delta x\rightarrow 0^\pm} u(x+\Delta x)$.  We
also denote
$$
\left\{\begin{array}{l}
         u_\iR=u(x_\iR)\,,
         \\[0.9em]
         u^\pm_\iR=u(x^\pm_\iR)\,.
\end{array}\right.
$$

From these notations, we apply a semi-discrete discontinuous Galerkin
method for \eqref{cn} as follows. We look for an approximation $C_{n,h}(t,\cdot) \in V_h^k$, such that for any $\testR_n \in V_h^k$, we have
\beq
\label{dgcn}
\f{\dD}{\dD t}\int_{I_j}C_{n,h}\,\testR_n\,\dD
x\,+\,a^j_{n}(g_{n},\testR_n)-\f{\sqrt{n}}{v_{th}}\,\int_{I_j}E\,C_{n-1,h}\,\testR_n\,\dD
x
\,=\, 0,\quad 0\le n \le N_H-1\,,
\eeq
where  $a^j_{n}$ is defined as
\beq
\label{anh}
\left\{
  \begin{array}{l}
\ds a^j_{n}(g_{n},\testR_n) \,=\, -\int_{I_j}g_{n}\,\testR'_n \,
    \dD
    x\,+\,\hat{g}_{n,j+\f12}\,\testR^-_{n,j+\f12}-\hat{g}_{n,j-\f12}\,\testR^+_{n,j-\f12}\,,
    \\[1.1em]
\ds g_{n} \,=\,v_{th}\left(\sqrt{n+1}\,C_{n+1,h}\,+\,\sqrt{n}\,C_{n-1,h}\right)\,.
\end{array}\right.
\eeq
The numerical flux $\hat{g}_{n}$ in \eqref{anh} could be taken as any consistent numerical fluxes. As an example, the global Lax-Friedrichs flux is used, which is defined as
\beq
\label{lfflx}
\hat{g}_{n}\,=\,\f12\left[g^-_{n}+g^+_{n}-\alpha\,\left(C^+_{n,h}\,-\,C^-_{n,h}\right)\right]\,,
\eeq
with the numerical viscosity coefficient $\alpha=v_{th}\sqrt{N_H}$.

For the Poisson equation \eqref{ps}, we  introduce the potential function $\Phi(t,x)$, such that
$$
\left\{
\begin{array}{l}
\ds E\,=\,-\f{\partial \Phi}{\px}\,, \\[0.9em]
\ds\f{\partial E}{\px} \,=\, v_{th}\,C_0 - \rho_0\,.
\end{array}\right.
$$
Hence we get the one dimensional Poisson equation 
$$
-\f{\partial^2 \Phi}{\px^2} \,=\, v_{th}\,C_0 \,-\, \rho_0\,,
$$
for which we also define a local discontinuous Galerkin  scheme:  we
look for a couple $\Phi_{h}(t,\cdot), E_h(t,\cdot) \in V_h^k$, such
that for any $\testP$ and $\testE \in V_h^k$, we have 
\beq
\label{dgP}
\left\{
\begin{array}{l}
\ds+\int_{I_j}\Phi_h\,\testP'\,\dD x \, -\,
  \hat{\Phi}_{h,j+\f12}\,\testP^-_{j+\f12} \,+\,\hat{\Phi}_{h,j-\f12}\,\testP^+_{j-\f12}\,=\,\int
  E_h\,\testP\,\dD x\,,
  \\[1.1em]
\ds -\int_{I_j}E_h\,\testE'\,\dD x \,+\,
  \hat{E}_{h,j+\f12}\,\testE^-_{j+\f12}
  \,-\,\hat{E}_{h,j-\f12}\,\testE^+_{j-\f12}\,=\,\int \left(
 v_{th}\, C_{0,h}\,-\,\rho_0\right)\,\testE\,\dD x\,,
\end{array}\right.
\eeq
where the numerical fluxes $\hat{\Phi}_h$ and $\hat{E}_h$ in
\eqref{dgP}  here are taken as in \cite{ayuso2009,ayuso2012},
\beq
\label{fluxps}
\left\{
\begin{array}{l}
  \hat{\Phi}_h \,=\, \{\Phi_h\}\,, \\[0.9em]
  \hat{E}_h \,=\, \{E_h\}\,-\,\beta\,[\Phi_h]\,,
\end{array}\right.
  \eeq
with $\beta$ either being a positive constant or a constant multiplying by $1/h$ and $h$ is the mesh size.

\subsection{Conservation properties for the semi-discrete scheme}
\label{sec:3.2}
The local discontinuous Galerkin method presented in the previous
section naturally preserves the mass since the equation on $C_0$ only
contains a convective term. Unfortunately, the momentum and total energy
are not conserved due to the contribution of the source
terms in \eqref{dgcn}. Therefore, we propose a slight modification of the fluxes for
the unknowns $(C_1,C_2)$ in order to ensure the right conservations.

In this subsection, without any confusion, we will drop the subindex
$h$ for simplicity. The aim is to  modify the scheme \eqref{dgcn}-\eqref{fluxps}
in order to ensure the conservation of mass, momentum and energy
without deteriorating the order of accuracy. Thus, we follow the proof
of Proposition \ref{prop:1} and consider the scheme \eqref{dgcn} for
$n=1,2$ combined with \eqref{dgP}, by adding two residual terms to ensure the
conservation of momentum and total energy
\beq
\label{dgc0}
\left\{
  \begin{array}{l}
\ds\f{\dD }{\dD t}\int_{I_j}C_1\,\testR_1\,\dD x\,+\,a^j_1(g_1,\testR_1)
    \,-\,\f{1}{v_{th}}\,\int_{I_j}E\,C_0\,\testR_1\,\dD x \,=\, \langle
    r^j_{1}, \,\phi_1\rangle\,,
  \\[0.9em]
\ds\f{\dD}{\dD t}\int_{I_j}C_2\,\testR_2\,\dD x\,+\,a^j_2(g_2,\testR_2)\,-\,\f{\sqrt{2}}{v_{th}}\,\int_{I_j}E\,C_1\,\testR_2\,\dD
    x
    \,=\, \langle  r^j_2, \phi_2 \rangle\,,
\end{array}\right.
\eeq
and the residual terms  $(r_1,r_2)$ are given by
\beq
\label{dgc0-r}
\left\{
  \begin{array}{l}
    \langle r^j_1,\testR_1\rangle \,=\, -\ds\frac{\beta}{2\,v_{th}^2}\,\left( [\Phi]\,[E]_{j-\f12}\, \testR^+_{1,j-\f12}\,+\, [\Phi]\,[E]_{j+\f12}\,\testR^-_{1,j+\f12}\right)\,,
\\[1.1em]
   \langle r^j_2,\testR_2\rangle \,=\, \ds\frac{1}{\sqrt{2}\,v_{th}}\,\left(
    \left(\{C_1\}-\hat{C}_1\right)\,[\Phi]_{j-\f12}\,\testR^+_{2,j-\f12} \,+\, \left(\{C_1\}-\hat{C}_1\right)\,[\Phi]_{j+\f12}\, \testR^-_{2,j+\f12}\right)\,. 
\end{array}\right.
\eeq
Apparently, these residues are given as products of two cell
interface jump terms, which both are of high order, hence it should
not deteriorate the order of accuracy. Therefore,  we prove that this
choice will ensure the  discrete  conservation properties of the scheme.

\begin{prop}
  \label{prop:2}
  For any {$N_H \ge 3 \in \NN$}, we consider the solution $(C_n,E,\Phi)$ to the
  system \eqref{dgcn}-\eqref{fluxps} combined with the residual terms
  \eqref{dgc0}-\eqref{dgc0-r}. Then, discrete mass and momentum are
  conserved
  $$
v_{th}^{k+1}\,\frac{\dD }{\dD t}\int_0^L C_k \, \dD x \, =\, 0, \quad k=0, \,1
$$
and the discrete total energy,  defined as
$$
\mE_h(t) \,:=\, \frac{1}{2}\,\int_{0}^L \, v_{th}^3 {\left(\sqrt{2} \,C_2
 + C_0 \right)} \,+\, |E|^2 \dD x  \,+\, \frac12 \,\beta\, \sum_{j}[\Phi]_{j-\f12}^2\,,
$$
is also preserved with respect to time.
\end{prop}
\begin{proof}
On the one hand, the conservation of mass easily follows by choosing $\phi_0=1$
in  \eqref{dgcn} for $n=0$ and summing over $j$, hence since there is no source term we get
	$$
	\sum_j \f{\dD}{\dD t}\int_{I_j} C_0\,\dD x \,=\, 0\,.
	$$
        On the other hand, for the conservation of momentum, taking
$\phi_1=1$ in the modified scheme \eqref{dgc0}  and summing over
$j$,  we have
\beq
\label{co:0}
	\f{\dD}{\dD t}\int_0^LC_1\,\dD
        x\,=\,\f{1}{v_{th}}\,\int_0^LE\,C_0\,\dD x \,+\, \sum_j
        \langle r_1^j, 1\rangle\,,
	\eeq
        where the contribution of the residual term, using periodic boundary
        conditions, gives
        \beq
          \sum_j \langle r_1^j,1\rangle \,=\, -\sum_j  \ds\frac{\beta}{2\,v_{th}^2}\,\left( [\Phi]\,[E]_{j-\f12}\,+\, [\Phi]\,[E]_{j+\f12}\right)
        \, =\, -\ds\frac{\beta}{v_{th}^2}\,\sum_j  
                [\Phi]\,[E]_{j-\f12}\,.
                \label{co:1}
        \eeq
	Then we compute the contribution of the first source term on the
        right hand side of \eqref{co:0}. Choosing $\testP=1$ in the first equation of
        \eqref{dgP}, we get that
        $$
\int_{I_j}E \,\dD x \,=\, -\hat{\Phi}_{j+\f12}+\hat{\Phi}_{j-\f12}\,, 
$$
whereas taking $\testE=E$ in the second equation of \eqref{dgP},   we have
	$$
	-\int_{I_j}E\, \f{\partial E}{\partial x}\,\dD x \,+\,
        \hat{E}_{j+\f12}\,E^-_{j+\f12}\,-\,\hat{E}_{j-\f12}\,E^+_{j-\f12}\,=\,\int_{I_j}(v_{th}\,C_0\,-\,\rho_0)\,E\,\dD
        x\,,
	$$
       hence adding  the latter both equalities, it yields 
	\begin{align*}
	v_{th}\,\int_{I_j}E\,C_0\,\dD x&=\,-\f12\int_{I_j}\f{\partial
                                         E^2}{\partial x}\,\dD
                                         x\,+\,
                                         \hat{E}_{j+\f12}E^-_{j+\f12}\,-\,\hat{E}_{j-\f12}E^+_{j-\f12}\,+\,\rho_0\,\int_{I_j}E\,\dD
                                         x\, \\
	&=\,-\f12(E^2)^-_{j+\f12}\,+\,\f12(E^2)^+_{j-\f12}\,+\, \hat{E}_{j+\f12}E^-_{j+\f12}\,-\,\hat{E}_{j-\f12}E^+_{j-\f12}\,+\,\rho_0\,(\hat{\Phi}_{j+\f12}-\hat{\Phi}_{j-\f12})\,.
	\end{align*}
   Summing on $j$, we obtain
   \beq
   \label{co:2}
   v_{th}\,\int_0^LE\,C_0\,\dD x\,=\,\sum_j\left(\f12[E^2]\,-\,\hat{E}\,[E]\right)_{j-\f12}\,=\,\sum_j\left(\{E\}-\hat{E}\right)_{j-\f12}\,[E]_{j-\f12}\,.
   \eeq
Finally, using the definition of the flux $\hat{E}$  in
\eqref{fluxps}, we have
   $$
   \left(\{E\}-\hat{E}\right)_{j-\f12}\,[E]_{j-\f12}\,=\,\beta\,[\Phi]\,[E]_{j-\f12}\,,
   $$ 
hence gathering \eqref{co:1} and \eqref{co:2} in \eqref{co:0}, we get the
conservation of momentum
$$
\f{\dD}{\dD t}\int_0^LC_1\,\dD
        x\,=\, 0\,.
$$

 Now let us show the conservation of total energy. Choosing
 $\phi_2=1$  in the modified scheme \eqref{dgc0}   and summing over
$j$,  we get 
    \beq
   \label{mo:0}
   \f{\dD}{\dD t}\int_0^L C_2\,\dD
   x\,=\,\f{\sqrt{2}}{v_{th}}\,\int_0^L E\,C_1\,\dD x \,+\,
   \sum_{j}\langle r_2^j, 1\rangle\,,
   \eeq
      where the contribution of the residual term, using periodic boundary
        conditions, gives
\beq
        \sum_j \langle r_2^j,1\rangle \,=\, \ds\frac{\sqrt{2}}{v_{th}}\, \sum_j      \left(\left(\{C_1\}-\hat{C}_1\right)\,[\Phi]\right)_{j-\f12}\,.
 \label{mo:1}
 \eeq
 Let us evaluate the first term on the right hand side of
 \eqref{mo:0}. We consider the  first equation of \eqref{dgP} and take $\testP=C_1$, 
$$
   \int_{I_j}E\,C_1\,\dD x\,=\,\int_{I_j}\Phi\,\f{\partial
     C_1}{\partial x}\,\dD x\,-\,\hat{\Phi}_{j+\f12}(C_1)^-_{j+\f12}\,+\,\hat{\Phi}_{j-\f12}(C_1)^+_{j-\f12}\,,
 $$
   whereas in \eqref{dgcn}, we choose $\phi_0=\Phi/v_{th}$, 
   $$
   \f{1}{v_{th}}\,\int_{I_j}\f{\partial C_0}{\partial t}\,\Phi\,\dD x \,=\, \int_{I_j}C_1\,\f{\partial \Phi}{\partial
     x}\,\dD x\,-\,\hat{C}_{1,j+\f12}\,\Phi^-_{j+\f12}\,+\,\hat{C}_{1,j-\f12}\,\Phi^+_{j-\f12}\,.
   $$
  Adding the latter two equalities, it yields 
   \begin{align*}
   &\int_{I_j}E\,C_1\,+\,\f{1}{v_{th}}\,\f{\partial C_0}{\partial t}\,\Phi\,\dD
     x\,
     \\[0.9em]
     &\,=\, \int_{I_j}\f{\partial  (\Phi \, C_1)}{\partial x}\,\dD
           x \;-\,\hat{\Phi}_{j+\f12}\,(C_1)^-_{j+\f12}\,+\,\hat{\Phi}_{j-\f12}\;(C_1)^+_{j-\f12}\,-\,\hat{C}_{1,j+\f12}\,\Phi^-_{j+\f12}\,+\,\hat{C}_{1,j-\f12}\,\Phi^+_{j-\f12}\,
     \\[0.9em] 
         &\,=\,(\Phi \,C_1)^-_{j+\f12}-(\Phi \,C_1)^+_{j-\f12} \,-\,\hat{\Phi}_{j+\f12}\,(C_1)^-_{j+\f12}\,+\,\hat{\Phi}_{j-\f12}\,(C_1)^+_{j-\f12} \,-\,\hat{C}_{1,j+\f12}\;\Phi^-_{j+\f12}\,+\,\hat{C}_{1,j-\f12}\,\Phi^+_{j-\f12}\,. 
   \end{align*}
   Summing on $j$ and using periodic boundary conditions, we get
   $$
      \int_0^L \left[ E\,C_1\,+\, \f{1}{v_{th}}\,\f{\partial C_0}{\partial t}\,\Phi\right]\,\dD x\,=\,-\sum_j\left([\Phi\; C_1]\,-\,\hat{\Phi}\,[C_1]\,-\,\hat{C}_1\,[\Phi]\right)_{j-\f12}\,.
   $$
  To evaluate the right hand side of the latter term, we use that
   $$
[\Phi C_1]\,-\,\hat{\Phi}\,[C_1]\,-\,\hat{C}_1\,[\Phi] \,=\,\left(\{\Phi\}-\hat{\Phi}\right)\,[C_1]\,+\,\left(\{C_1\}\,-\,\hat{C}_1\right)\,[\Phi]\,,
   $$
   and from the definition of the flux for the Poisson equation
   \eqref{fluxps}, we get 
   \beq
   \label{mo:2}
   \f{\sqrt{2}}{v_{th}}\,\int_0^L \left[ E\,C_1\,+\, \f{1}{v_{th}}\,\f{\partial C_0}{\partial t}\,\Phi\right]\,\dD x\,=\,-\f{\sqrt{2}}{v_{th}}\,\sum_j\left(\left(\{C_1\}\,-\,\hat{C}_1\right)\,[\Phi]\right)_{j-\f12}\,.
   \eeq
   Now, it remains to evaluate the second term on the left hand side
   in \eqref{mo:2},   hence we compute the time derivative   of the second equation in
   \eqref{dgP} and choose  the test function $\testE=\sqrt 2\,\Phi/v_{th}^3$, it gives
$$
   \f{\sqrt{2}}{v_{th}^2}\,\int_{I_j}\f{\partial C_0}{\partial t}\,\Phi\,\dD x \,=\,
  \f{\sqrt{2}}{v_{th}^3}\,\left[-\int_{I_j}\f{\partial E}{\partial t}\,\f{\partial\Phi}{\partial
     x}\,\dD x\,+\, \left(\widehat{\f{\partial E}{\partial t}}\right)_{j+\f12}\,\Phi^-_{j+\f12}\,-\,\left(\widehat{\f{\partial
         E}{\partial t}}\right)_{j-\f12}\,\Phi^+_{j-\f12}\right]\,.
 $$
   Then  in the first equation of \eqref{dgP}, we take
  $\testP=-\f{\sqrt 2}{v_{th}^3}\f{\partial E}{\partial t}$ and obtain
\begin{eqnarray*}
 -\f{1}{\sqrt 2\,v_{th}^3}\,\f{\dD}{\dD
     t}\int_{I_j}E^2\,\dD x &=&-\f{\sqrt 2}{v_{th}^3}\int_{I_j}E
                                \,\f{\partial E}{\partial t}\,\dD x
  \\
&=&  -\f{\sqrt 2}{v_{th}^3}\left( \int_{I_j}\Phi\,\f{\partial^2 E}{\partial t\partial x}\,\dD
   x\,-\,\hat{\Phi}_{j+\f12}\,\left(\f{\partial E}{\partial
     t}\right)^-_{j+\f12}\,+\,\hat{\Phi}_{j-\f12}\,\left(\f{\partial E}{\partial
  t}\right)^+_{j-\f12}\right)\,.
 \end{eqnarray*}
 Adding the latter two equalities and summing over $j$, we get
 $$ 
   -\f{\sqrt 2}{v_{th}^2}\,\left[\f{1}{2\,v_{th}}\,\f{\dD}{\dD t}\int_0^L E^2\dD x \,-\,\int_0^L\f{\partial
     C_0}{\partial t}\,\Phi\,\dD
   x\right]\,=\,\f{\sqrt 2}{v_{th}^3}\,\sum_j\left(\left[\Phi \f{\partial
         E}{\partial t}\right]\,-\,\widehat{\f{\partial E}{\partial t}}\,[\Phi]\,-\,\hat{\Phi}\left[\f{\partial E}{\partial
         t}\right] \right)_{j-\f12}\,.
  $$
   Again using that,
   \begin{align*}
\left[\Phi \f{\partial E}{\partial
    t}\right]\,-\,\widehat{\f{\partial E}{\partial t}}\,[\Phi] \,-\,\hat{\Phi}\,\left[\f{\partial E}{\partial
    t}\right]
    &\,=\,\left(\{\Phi\}-\hat{\Phi}\right)\left[\f{\partial E}{\partial
      t}\right]\,+\,\left(\left\{\f{\partial E}{\partial
      t}\right\}\,-\,\widehat{\f{\partial E}{\partial t}}\right)\,[\Phi]\,
     \\
     &\,=\, \beta \,\left[\f{\partial \Phi}{\partial t}\right][\Phi]\,
       \\
     &\,=\,\f{\beta}{2} \,\f{\dD}{\dD t} [\Phi]^2\,,
   \end{align*}
   we get the time derivative of the discrete potential energy,
   \beq
   \label{mo:3}
 \f{1}{\sqrt{2}\,v_{th}^3}\,\f{\dD}{\dD t}
   \left[\int_0^L E^2\dD x + \beta\,\sum_j [\Phi]_{j-\f12}^2\right] \,=\, \f{\sqrt 2}{v_{th}^2}\,\int_0^L\f{\partial
     C_0}{\partial t}\,\Phi\,\dD
   x\,.
   \eeq
Gathering \eqref{mo:1}, \eqref{mo:2} and \eqref{mo:3} in \eqref{mo:0}, we get
 $$
   \f{\dD}{\dD t}\left[\int_0^L \left(C_2 \,+\;\f{1}{\sqrt 2
         \,v_{th}^3}\,E^2\right)\,\dD x \,+\, \f{\beta}{\sqrt 2
         \,v_{th}^3}\,\sum_{j} [\Phi]^2_{j-\f12} \right] \,=\, 0\,.
     $$
     Finally, using the conservation of mass and the latter result, we
     get the conservation of discrete energy
     $$
 \f12\,\f{\dD}{\dD t}\left[\int_0^L
         \,v_{th}^3\,{\left( \sqrt 2 C_2 \,+\; C_0\right)}  \,+\,
         |E|^2\,\dD x \,+\, \beta\,\sum_{j} [\Phi]^2_{j-\f12} \right] \,=\, 0\,.
     $$
 \end{proof}

 \begin{rem}
  \label{rem:3.1}
It is possible to apply the same method considering  the symmetrically
weighted basis with the $L^2(\RR)$ norm in order to design a
conservative discontinuous Galerkin method. However, the residue
$(r_1,r_2)$ would involve all the modes $(C_n)_{0\leq n< N_H}$ causing
an additional computational cost. This is the main reason why we favor
the asymmetrically weighted basis.
\end{rem}

\subsection{Time discretization}
\label{sec:3.3}
In this section, we propose a  time discretization to the modified
semi-discrete scheme \eqref{dgcn}-\eqref{fluxps}, with
\eqref{dgc0}-\eqref{dgc0-r} to ensure the conservation of momentum and
energy. As it has been shown in the proof of Proposition \ref{prop:2},
the conservation of momentum can be obtained by a simple modification
of the source term, whereas the conservation of energy is more tricky
since it involves time derivative. Therefore,  the time discretization
is an issue for the conservation of energy.

We denote  by $\bC=(C_0,\ldots,C_{N_H-1})$ the solution to
\eqref{dgcn}-\eqref{dgc0-r} and by $(\cdot,\cdot)$ the standard $L^2$ inner
product on the interval $(0,L)$
$$
(C_n\,,\;\phi)\,:=\,\int_0^L C_n\,\phi\,\dD x
$$
and
 $(\cdot,\cdot)_I$ the standard $L^2$ inner
product on the interval $I$
$$
(C_n\,,\;\phi)_I\,:=\,\int_I C_n\,\phi\,\dD x\,.
$$
Then for the computation of the source terms, we introduce $b_n^j$
given by, 
$$
b_n^j(\bC,E,\Phi,\phi_n) \,=\,
\left\{
  \begin{array}{ll}
    0\,, & \textrm{\,if \,} n = 0\,,
    \\[1.1em]
    \ds-\f{1}{v_{th}}\,(E\,C_0\,,\;\phi_1)_{I_j} \,-\,  \langle r^j_1,\testR_1\rangle\,, & \textrm{\,if \,} n = 1\,,
   \\[1.1em]
    \ds-\f{\sqrt 2}{v_{th}}\,(E\,C_1\,,\;\phi_2)_{I_j}\,-\,  \langle r^j_2,\testR_2\rangle\,, & \textrm{\,if \,} n = 2\,,
   \\[1.1em]
    \ds-\f{\sqrt n}{v_{th}}\,(E\,C_{n-1}\,,\;\phi_n)_{I_j}\,, &
                                                                   \textrm{\,if
                                                                   \,}
                                                                   3\leq
                                                                   n
                                                                   \leq
                                                                   N_H-1\,,
    \end{array}\right.
$$
where $(r_1,r_2)$ is defined in \eqref{dgc0-r}. We set {$a_n=\sum_j a_n^j$
and $b_n=\sum_j b_n^j$}. Using these notations,
 the semi-discrete system \eqref{dgcn}-\eqref{dgc0-r} can be written
 in the compact form for the time evolution of $(C_n)_{0\leq n< N_H}$,
 \beq
 \label{dgcn-new}
\f{\dD }{\dD t} (C_n,\phi_n) \,+\, a_n(g_n,\phi_n) \,+\,
b_n(\bC,E,\Phi,\phi_n)\, =\, 0, \quad \forall \; \phi_n\,\in\,  V_h^k\,,
 \eeq
coupled with the local discontinuous Galerkin scheme for the Poisson
equation \eqref{dgP}-\eqref{fluxps}.

We now present a first and second order time discretizations to
\eqref{dgcn-new} preserving the conservation properties. For
simplicity we fix an interval $I$, for which we
drop the index $j$ and choose a time step $\Delta t$. We compute an
approximation $\bC^m=(C_0^m,\ldots,C_{N_H-1}^m)$ of the solution $\bC$ at time $t^m\,=\,m\,\Delta t$, for
$m\geq 0$.

\subsubsection*{First order scheme}

We  first  apply a classical explicit Euler scheme
for any $n\in\{0,\ldots,\,N_H-1\}\setminus\{2\}$ and $m\geq 0$,
\beq
\label{euler-1}
\f{(C_n^{m+1}-C_n^{m},\phi_n)}{\Delta t} \,+\,a_n(g_n^m,\phi_n) \,+\,
b_n(\bC^m,E^m,\Phi^m,\phi_n)\, =\, 0, \quad \forall \; \phi_n\,\in\,
V_h^k\,.
\eeq
From this system, we get the discrete density $C_0^{m+1}$ at time
$t^{m+1}$ and can apply the local discontinuous Galerkin scheme
\eqref{dgP}-\eqref{fluxps} for
the computation of the electric field and potential
$(E^{m+1},\Phi^{m+1})$, hence we compute
$$
E^{m+\f12} = \f12\,\left(E^m+E^{m+1}\right)\,, \quad \Phi^{m+\f12} = \f12\,\left(\Phi^m+\Phi^{m+1}\right)\,, 
$$
and the unknown $C_2^{m+1}$, given by
\beq
\label{euler-2}
\f{ (C_2^{m+1}-C_2^{m},\phi_2)}{\Delta t} \,+\,a_2(g_2^m,\phi_2) \,+\,
b_2(\bC^m,E^{m+\f12},\Phi^{m+\f12},\phi_n)\,=\, 0, \quad \forall \; \phi_2\,\in\,
V_h^k\,.
\eeq

Now let us prove mass, momentum and energy conservation for this fully discrete scheme, under periodic or compact support boundary conditions.
\begin{prop}
  \label{prop:3}
  For any {$N_H \ge 3 \in \NN$}, we consider the solution $(\bC^m,E^m,\Phi^m)_m$ to the
  system \eqref{euler-1}-\eqref{euler-2}. Then, discrete mass and momentum are
  conserved for any $m\geq 0$,
  $$
v_{th}^{k+1}\,\int_0^L C_k^m \, \dD x \, =\, v_{th}^{k+1}\,\int_0^L C_k^0 \, \dD x , \quad k=0, \,1
$$
and the discrete total energy,  defined as
$$
\mE_h^m \,:=\, \frac{1}{2}\,\int_{0}^L \, v_{th}^3 {\left(\sqrt{2}\, C_2^m
 + C_0^m\right)} \,+\, |E^m|^2 \dD x  \,+\,\frac12\,\beta\, \sum_{j}[\Phi^m]_{j-\f12}^2\,,
$$
is also preserved with respect to $m\geq 0$.
\end{prop}
\begin{proof}
  The conservation of mass and momentum follows the lines of the proof
  in Proposition \ref{prop:2}, since we have by construction
  $$
\sum_j a_0^j(g_0^m,1) \,=\, 0\,,\quad{\rm and}\quad  \sum_j
a_1^j(g_1^m,1) \,=\, \sum_j b_1^j(\bC^m,E^m,\Phi^m,1)\,=\, 0\,.
$$

For the conservation of the total energy,  we also proceed as in
Proposition \ref{prop:2}.  We first have
$$
\sum_j a_2^j(g_0^m,1) \,=\, 0\,,
$$
and also
$$
\sum_j b_2^j(\bC^m,E^{m+\f12},\Phi^{m+\f12},1)\,=\, -\f{\sqrt
  2}{v_{th}} \left( (E^{m+\f12}\,,\,C_1^m) \,+\, \sum_{j} \left(\{C_1^m\}-\hat{C}_1^m\right)\,[\Phi^{m+\f12}]_{j-\f12}\right)\,.
$$
Then using the scheme for the Poisson equation \eqref{dgP} with
$\testP=C_1^m$ and the scheme \eqref{euler-1} for $C_0^{m+1}$ with
$\phi_0=\Phi^{m+\f12}/v_{th}$, we obtain
$$
{\f{ (C_2^{m+1}-C_2^{m},1)}{\Delta t} \,=\,}-\sum_j b_2^j(\bC^m,E^{m+\f12},\Phi^{m+\f12},1)\,=\, - \f{\sqrt
  2}{v_{th}^2} \int_0^L \f{C_0^{m+1}-C_0^m}{\Delta t} \, \Phi^{m+\f12}
\,\dD x. 
$$
On the other hand by linearity of the scheme \eqref{dgP} for the
Poisson equation, we apply the scheme to $\f{C_0^{m+1}-C_0^m}{\Delta t}$
and choose  the test function $\testE=\sqrt 2\,\Phi^{m+\f12}/v_{th}^3$
in the second equation  and $\testP=-\f{\sqrt
  2}{v_{th}^3}\f{E^{m+1}-E^{m}}{\Delta t}$ in the first equation, it gives the variation of
the discrete potential energy
$$
\f{\sqrt
     2}{v_{th}^2}\,\int_0^L\f{C_0^{m+1}-C_0^m}{\Delta t}\,\Phi^{m+\f12}\,\dD
   x\,=\;\f{1}{\sqrt{2}\,v_{th}^3}\,
   \left[\int_0^L \f{|E^{m+1}|^2-|E^m|^2}{\Delta t} \dD x \,+\,\beta \, \sum_j
     \f{[\Phi^{m+1}]_{j-\f12}^2 -[\Phi^{m}]_{j-\f12}^2}{\Delta t}
   \right] \,.
   $$
   Finally gathering the previous result, it yields
   \begin{eqnarray*}
   &&\int_0^L C_2^{m+1}\,+\, \f{1}{\sqrt{2}\,v_{th}^3}\,|E^{m+1}|^2\dD x \,+\, \f{\beta}{\sqrt{2}\,v_{th}^3}\,\sum_j
      [\Phi^{m+1}]_{j-\f12}^2 \\
     &&=\, \int_0^L C_2^{m}\,+\, \f{1}{\sqrt{2}\,v_{th}^3}\,|E^{m}|^2\dD x \,+\, \f{\beta}{\sqrt{2}\,v_{th}^3}\,\sum_j
     [\Phi^{m}]_{j-\f12}^2
   \end{eqnarray*}
and from the conservation of mass, we get the conservation of
numerical total energy
\begin{eqnarray*}
   &&\f{1}{2}\left[\int_0^L \,v_{th}^3 {(\sqrt{2} \, C_2^{m+1}+C_0^{m+1})}\,+\, |E^{m+1}|^2\dD x \,+\,\beta \,\sum_j
  [\Phi^{m+1}]_{j-\f12}^2\right] \\
  &&\,=\, \f{1}{2}\left[\int_0^L \,v_{th}^3 {(\sqrt{2} \,C_2^{m}+C_0^{m})}\,+\, |E^{m}|^2\dD x \,+\,\beta \, \sum_j
     [\Phi^{m}]_{j-\f12}^2\right].
    \end{eqnarray*}
\end{proof}

\subsubsection*{Second order scheme}

Let us extend the previous scheme to  second order.  We  first  apply
a classical explicit Euler scheme with a half-time step $\Delta t/2$ for any $n\in\{0,\ldots,\,N_H-1\}\setminus\{2\}$ and $m\geq 0$,
\beq
\label{rk-1}
\f{2\,(C_n^{(1)}-C_n^{m},\phi_n)}{\Delta t} \,+\,a_n(g_n^m,\phi_n) \,+\,
b_n(\bC^m,E^m,\Phi^m,\phi_n)\, =\, 0, \quad \forall \; \phi_n\,\in\,
V_h^k\,.
\eeq
From this system, we get the discrete density $C_0^{(1)}$ at time
$t^{m+\f12}=t^m+\Delta t/2$ and can apply \eqref{dgP}-\eqref{fluxps} for
the computation of the electric field and potential
$(E^{(1)},\Phi^{(1)})$ with the source term $C_0^{(1)}$, hence we compute the unknown $C_2^{(1)}$, given by
\beq
\label{rk-2}
\f{2\, (C_2^{(1)}-C_2^{m},\phi_2)}{\Delta t} \,+\,a_2(g_2^m,\phi_2) \,+\,
b_2\left(\bC^m,E^{m+\f14},\Phi^{m+\f14},\phi_n\right)\,=\, 0, \quad \forall \; \phi_2\,\in\,
V_h^k\,,
\eeq
with
$$
E^{m+\f14} = \f12\,\left(E^m+E^{(1)}\right)\,, \quad \Phi^{m+\f14} = \f12\,\left(\Phi^m+\Phi^{(1)}\right)\,.
$$
Then we compute a second stage with a time step $\Delta t$ for any $n\in\{0,\ldots,\,N_H-1\}\setminus\{2\}$ and $m\geq 0$,
\beq
\label{rk-3}
\f{(C_n^{m+1}-C_n^{m},\phi_n)}{\Delta t} \,+\,a_n(g^{(1)}_n,\phi_n) \,+\,
b_n(\bC^{(1)},E^{(1)},\Phi^{(1)},\phi_n)\, =\, 0, \quad \forall \; \phi_n\,\in\,
V_h^k\,.
\eeq
From this system, we get the discrete density $C_0^{m+1}$ at time
$t^{m+1}=t^m+\Delta t$ and can apply the local discontinuous Galerkin scheme
\eqref{dgP}-\eqref{fluxps} for the computation of the electric field and potential
$(E^{m+1},\Phi^{m+1})$, hence we compute
$$
E^{m+\f12} = \f12\,\left(E^m+E^{m+1}\right)\,, \quad \Phi^{m+\f12} = \f12\,\left(\Phi^m+\Phi^{m+1}\right)\,, 
$$
and the unknown  $C_2^{m+1}$, given by
\beq
\label{rk-4}
\f{ (C_2^{m+1}-C_2^{m},\phi_2)}{\Delta t} \,+\,a_2(g_2^{(1)},\phi_2) \,+\,
b_2(\bC^{(1)},E^{m+\f12},\Phi^{m+\f12},\phi_n)\,=\, 0, \quad \forall \; \phi_2\,\in\,
V_h^k\,.
\eeq
We have the following result for the second order scheme
\eqref{rk-1}-\eqref{rk-4}.
\begin{prop}
  \label{prop:4}
  For any {$N_H \ge 3 \in \NN$}, we consider the solution $(\bC^m,E^m,\Phi^m)_m$ to the
  system \eqref{rk-1}-\eqref{rk-4}. Then, discrete mass and momentum are
  conserved for any $m\geq 0$,
  $$
v_{th}^{k+1}\,\int_0^L C_k^m \, \dD x \, =\, v_{th}^{k+1}\,\int_0^L C_k^0 \, \dD x , \quad k=0, \,1
$$
and the discrete total energy,  defined as
$$
\mE_h^m \,:=\, \frac{1}{2}\,\int_{0}^L \, v_{th}^3 {\left(\sqrt{2} C_2^m
 + C_0^m\right)} \,+\, |E^m|^2 \dD x  \,+\,\frac12 \,\beta \, \sum_{j}[\Phi^m]_{j-\f12}^2,
$$
is also preserved with respect to $m\geq 0$.
\end{prop}
\begin{proof}
As we can see, the first stage is the same as the first order scheme \eqref{euler-1}-\eqref{euler-2}, but with a half time step at $t^m+\Delta t/2$, so
the conservation of mass, momentum and energy is directly
obtained. For the second stage, the proof is very similar, and this
step directly ensures conservation of mass, momentum and energy.
\end{proof}

%%%%%%%%%%%%%%%%%%%%%%%%%%%%%%%%%%
%
%%%%%%%%%%%%%%%%%%%%%%%%%%%%%%%%%%
\section{Numerical examples}
\label{sec:4}
\setcounter{equation}{0}
\setcounter{figure}{0}
\setcounter{table}{0}

In this section, we will verify our proposed DG/Hermite Spectral method for the one-dimensional Vlasov-Poisson (VP) system. We take the phase domain along $v$ to be $[-8,8]$, with $N_H$ modes for Hermite spectral bases, and $N_x$ cells in space. We note that due to the Hermite spectral basis, there is no truncation error for the conservation of mass, momentum and energy from cut-off along the $v$-direction. The scaling velocity $v_{th}$ is chosen to be $1$ and the Hou-Li filter with $2/3$ dealiasing rule \cite{hou2007computing,filtered} will be used, if without specification.

In the following, we take $P_2$ piecewise polynomial in space and 2nd order scheme in time. We denote this scheme as ``DGHSM". We compute reference solutions using a positivity-preserving semi-Lagrangian finite difference (SLFD) WENO scheme in \cite{xiong2014high}, which is denoted as ``SLFDM". The SLFDM uses discrete velocity
coordinate, and the mesh size is $N_x\times N_v$.

\subsection{Landau damping}
\label{sec:4.1}
We first consider the Landau damping problem for the VP system
with the initial condition:
\beq
\label{landau}
f(0,x,v)=(1+\alpha \cos(k x))\f{1}{\sqrt{2\pi}}\exp(-\f{v^2}{2})\,.
\eeq
Under Hermite spectral basis in velocity, we have
\beq
\label{landauc}
C_0(t,x) = 1+\alpha \cos(k x)\,,
\eeq
and $C_n(t,x) = 0, n\ge1$. The background density $\rho_0=1$. The length of the domain in the $x$-direction is $L=\f{2\pi}{k}$. 

We first take $\alpha=0.05$ with $k=0.5$ and use this example to verify the orders of our scheme. We compute the solution to $T=0.1$. The errors are computed by comparing to the solution on using $N_x=512$ and $P_3$ piecewise polynomial bases, and $N_H=600$. In Table \ref{tab41}, we show the $L^2$ errors and orders for $P_k$ piecewise polynomials with $k=1,2,3$ respectively. Due to the time steps are smaller than the spatial mesh size, we can observe $(k+1)$-th order convergence for $P_k$ polynomials respectively.

We then consider the weak Landau damping with $\alpha=0.01$ and $k=0.5$ in \eqref{landau}. We take $N_H=256$ and $N_x=64$ and compute the solution up to $T=60$. We show the time evolution of the relative deviations of discrete mass, momentum and total energy in Figure \ref{fig41} (left). We can see the errors for these three conservative quantities are up to machine precision, which demonstrate that the scheme can preserve exact mass, momentum and energy. We plot the time evolution of the electric field in $L^2$ norm ($E_2$) and $L^\infty$ norm $(E_{max})$ in Figure \ref{fig41} (right). The numerical damping rate for this case matches the theoretical linear damping rate $-0.1533$ very well.  

We also consider the strong Landau damping with $\alpha=0.5$ and $k=0.5$ in \eqref{landau}. Here we take $N_H=1024$ and $N_x=64$ and compute the solution up to $T=60$. Time evolution of the relative deviations of discrete mass, momentum and total energy are reported in Figure \ref{fig42} (left). Similarly these errors are up to machine precision. We also plot the time evolution of the electric field in $L^2$ norm ($E_2$) and $L^\infty$ norm $(E_{max})$ in Figure \ref{fig42} (right). We compare the electric field to the one obtained by the SLFD WENO scheme with mesh size $256\times512$. We can see that the results match each other. We show the surface plots of the distribution function at $T=40$ for both methods, and they are also very similar.  

\begin{table}[!ht] 
	\centering
	\caption{Numerical $L^2$ errors and orders for Landau damping with initial distribution \eqref{landau}. $\alpha=0.05$ and $k=0.5$. $T=0.1$. $N_H=600$.}
	\label{tab41}
	\vspace{0.15in}
	\begin{tabular}{|l|l|l|l|l|l|l|} \hline
	\multicolumn{1}{|l|}{}&\multicolumn{2}{|c|}{$P_1$}&\multicolumn{2}{|c|}{$P_2$}&\multicolumn{2}{|c|}{$P_3$}\cr\cline{2-7}\hline		
	$N_x$ & $L^2$ error & Order & $L^2$ error & Order & $L^2$ error & Order  \\ \hline
	32 & 1.81E-4 & --   & 4.88E-6 & --   & 1.43E-07 & --   \\ \hline
	64 & 4.53E-5 & 2.00 & 5.65E-7 & 3.11 & 1.24E-08 & 3.53 \\ \hline
	128& 1.17E-5 & 1.95 & 6.22E-8 & 3.18 & 7.84E-10 & 3.98 \\ \hline
    256& 3.17E-6 & 1.88 & 6.94E-9 & 3.16 & 4.96E-11 & 3.98 \\ \hline		
	\end{tabular}
\end{table}

\begin{figure}[!ht] 
	\centering
	\includegraphics[width=3.2in,clip]{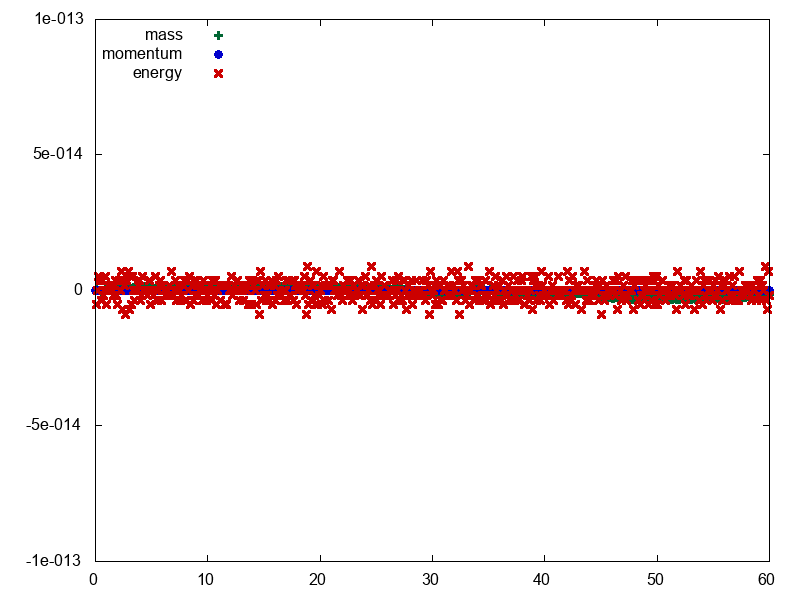},
	\includegraphics[width=3.2in,clip]{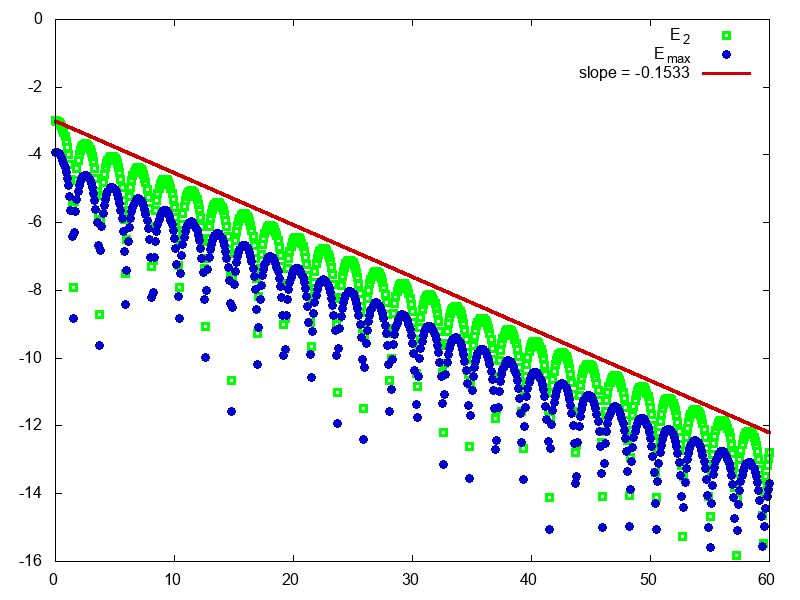}\\
\caption{Weak Landau damping of the initial distribution (\ref{landau}) with $\alpha=0.01$ and $k=0.5$. $T=60$. 
Left: deviation of mass, momentum and energy. Right: time evolution of the electric field in $L^2$ norm ($E_2$) and $L^\infty$ norm ($E_{max}$) (logarithmic value). Mesh size: $N_x\times N_H = 64 \times 256$.}
\label{fig41} 
\end{figure}

\begin{figure}
	\centering
	\includegraphics[width=3.2in,clip]{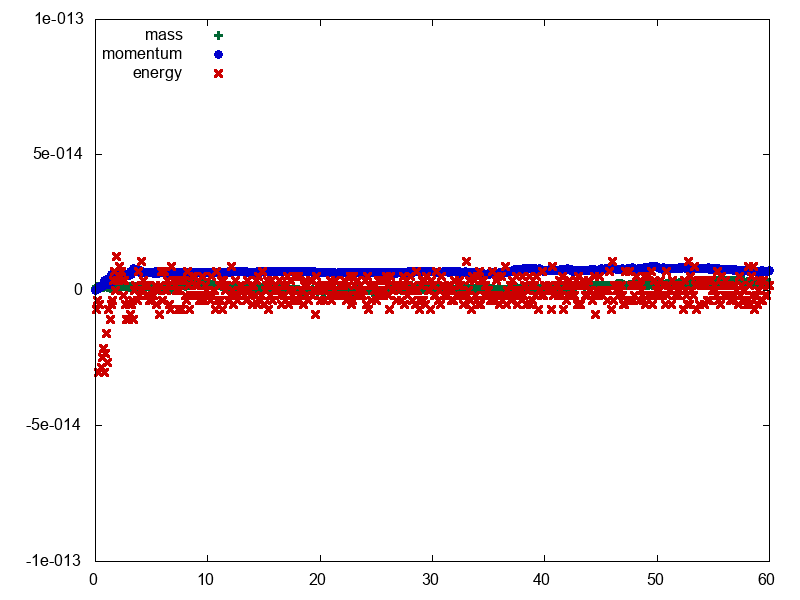},
	\includegraphics[width=3.2in,clip]{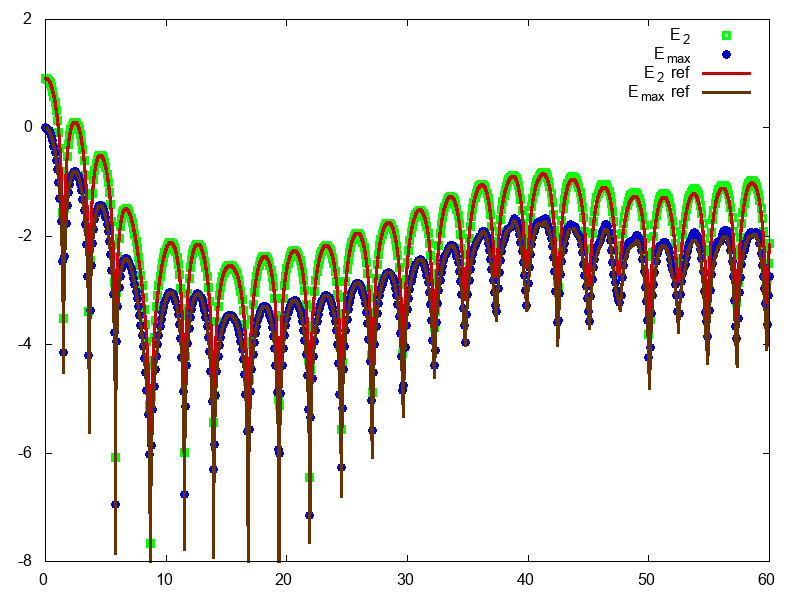}\\
	\caption{Strong Landau damping of the initial distribution (\ref{landau}) with $\alpha=0.5$. $T=60$. 
		Left: deviation of mass, momentum and energy. Right: time evolution of the electric field in $L^2$ norm ($E_2$) and $L^\infty$ norm ($E_{max}$) (logarithmic value). DGFSM: $N_x\times N_H = 64 \times 1024$. The reference solution is from SLFDM with $N_x\times N_v=256\times 512$.}
	\label{fig42}
\end{figure}

\begin{figure}
	\centering
	\includegraphics[width=3.2in,clip]{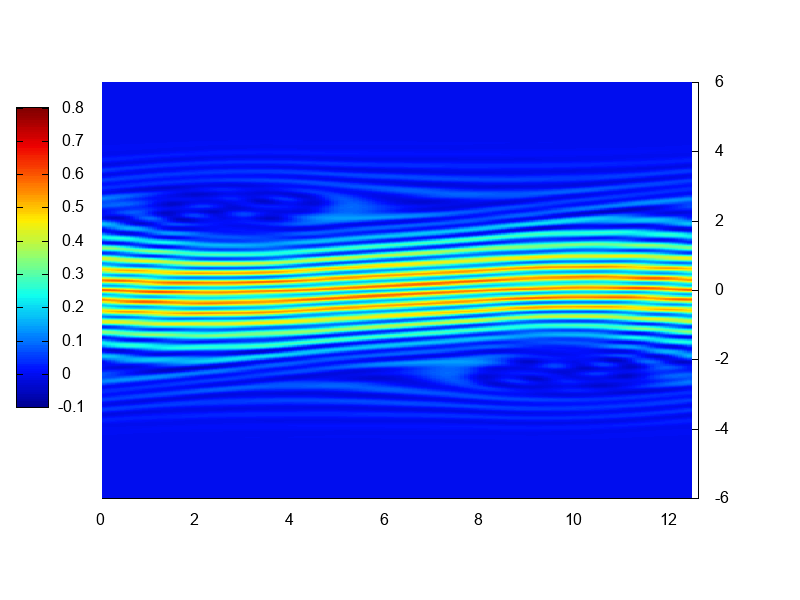}
	\includegraphics[width=3.2in,clip]{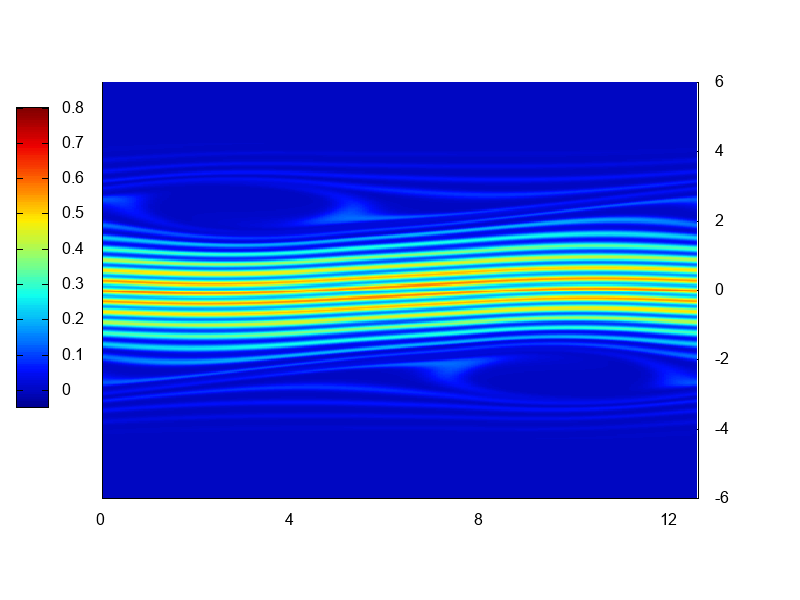}
	\caption{Strong Landau damping of the initial distribution (\ref{landau}) with $\alpha=0.5$ and $k=0.5$. Surface plot of the distribution function at $T=40$.  Left: $N_x\times N_H = 64 \times 1024$ for DGHSM; right: $N_x\times N_v = 256\times 512$ for SLFDM. }
	\label{fig43}
\end{figure}

\subsection{Two stream instability}
\label{sec:4.2}

In this example, we consider the two stream instability problem with the initial distribution function
\beq
\label{2stream}
f(0,x,v)=\f{2}{7}(1+5v^2)(1+\alpha((\cos(2kx)+\cos(3kx))/1.2+\cos(kx))\f{1}{\sqrt{2\pi}}\exp(-\f{v^2}{2})\,,
\eeq
where $\alpha=0.01$ and $k=0.5$. For this case, we have
$$
\left\{
  \begin{array}{l}
\ds C_0(t,x)\,=\,\f{12}{7}(1+\alpha((\cos(2kx)+\cos(3kx))/1.2+\cos(kx))\,,\\
\ds C_2(t,x)\,=\,\f{10\sqrt{2}}{7}(1+\alpha((\cos(2kx)+\cos(3kx))/1.2+\cos(kx))\,.
  \end{array}
\right.
$$
Other $C_n(t,x)$'s are all zero and $\rho_0=12/7$. Similarly $L=\f{2\pi}{k}$.

We compute the solution up to time $T=60$. For DGHSM, we take $N_x=64$ and $N_H=256$. We show the time evolution of the relative deviations of discrete mass, momentum and total energy in Figure \ref{fig44} (left). We can see these errors are still up to machine precision. Although the momentum decreases a little, the errors are at the level of $10^{-13}$. We also plot the time evolution of the electric field in $L^2$ norm ($E_2$) and $L^\infty$ norm $(E_{max})$ in Figure \ref{fig44} (right). We compare them with the results from SLFDM on the mesh $256\times512$, and they match each other very well. For this case, we show the surface plots of $f$ at $T=50$ in Figure \ref{fig45} for both methods. The results are comparable, only very weak high frequency instabilities are observed localized in the center for DGFSM.

\begin{figure}
	\centering
	\includegraphics[width=3.2in,clip]{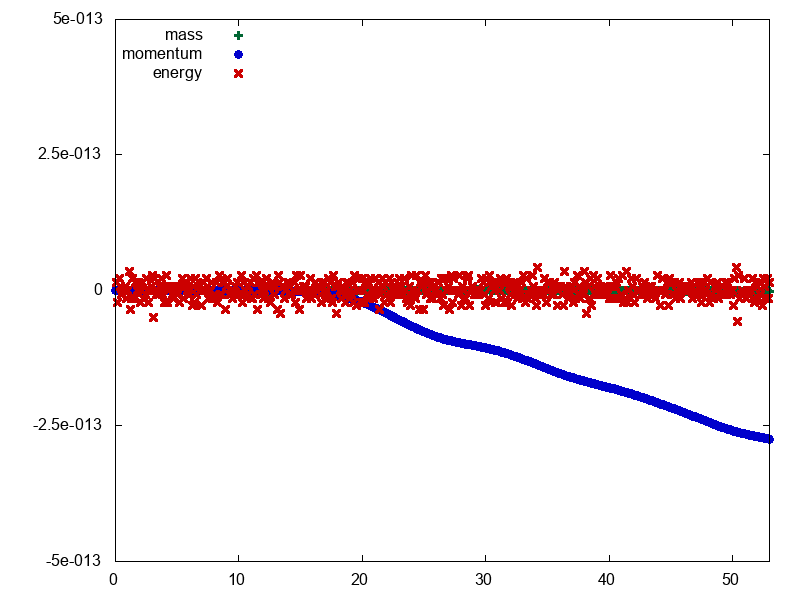},
	\includegraphics[width=3.2in,clip]{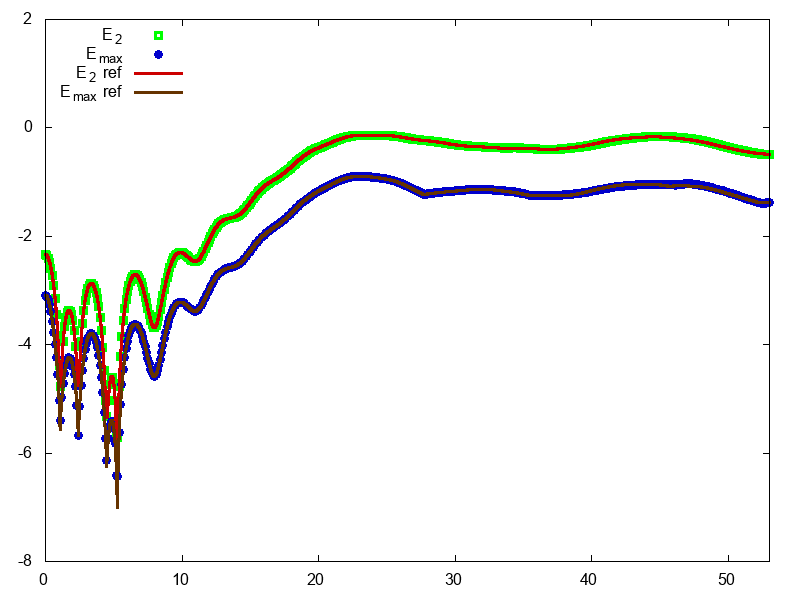}\\
	\caption{Two stream instability of the initial distribution \eqref{2stream} with $\alpha=0.01$ and $k=0.5$. Left: deviation of mass, momentum and energy. Right: time evolution of the electric field in $L^2$ norm ($E_2$) and $L^\infty$ norm ($E_{max}$) (logarithmic value). DGFSM: $N_x\times N_H = 128 \times 512$. The reference solution is from SLFDM with $N_x\times N_v=256\times 512$.}
	\label{fig44}
\end{figure}

\begin{figure}
	\centering
	\includegraphics[width=3.2in,clip]{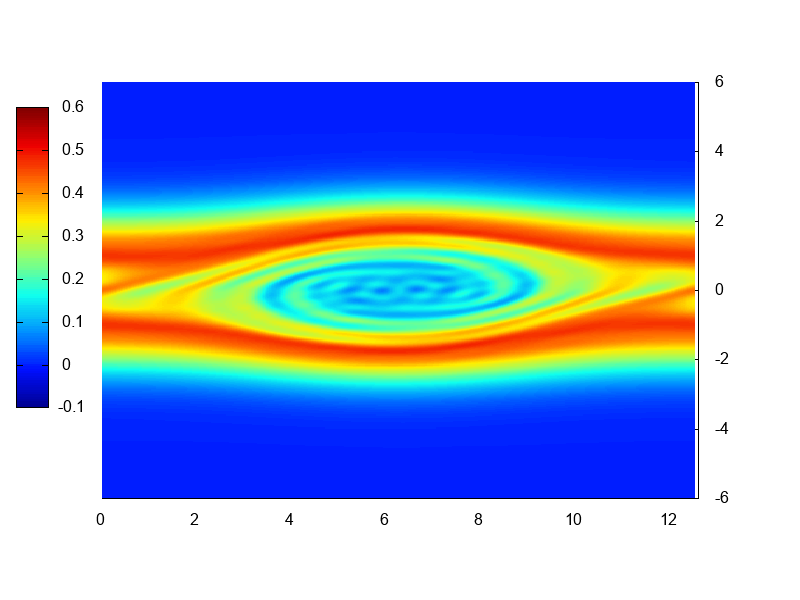}
	\includegraphics[width=3.2in,clip]{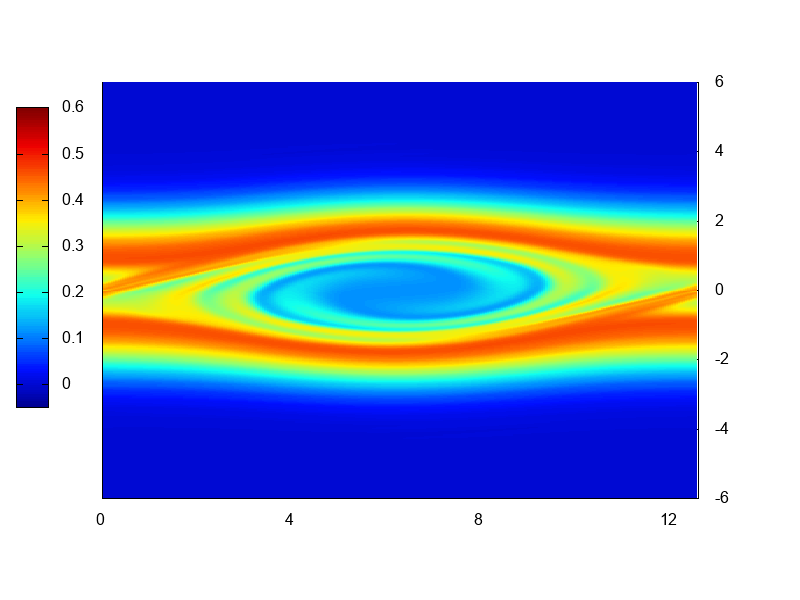}
	\caption{Two stream instability of the initial distribution (\ref{2stream}) with $\alpha=0.01$ and $k=0.5$. Surface plot of the distribution function $f$ at $T=50$. Left: $N_x\times N_H = 128 \times 512$ for DGHSM; right: $N_x\times N_v = 256\times 512$ for SLFDM. }
	\label{fig45}
\end{figure}

\subsection{Bump-on-tail instability}
\label{sec:4:3}
Lastly we consider the bump-on-tail instability problem with the initial distribution as
\begin{align}
\label{bot}
f(0,x,v)=f_{b.o.t}(v)(1+\alpha\cos(k\,n\,x))\,,
\end{align} 
where the bump-on-tail distribution is
\begin{align}
f_{b.o.t}(v)=\frac{n_p}{\sqrt{\pi}v_{th,p}}\exp\left(-\frac{v^2}{v^2_{th,p}}\right)+\frac{n_b}{\sqrt{\pi}v_{th,b}}\exp\left(-\frac{(v-v_{d,b})^2}{v_{th,b}^2}\right)\,.
\end{align}

We first choose the parameters to be $n_p=0.99$, $n_b=0.01$, $v_{d,b}=1$, $v_{th,p}=0.28284271$, $v_{th,b}=7.0710678$E$-02$, $\alpha=0.0001$, $n=10$ and $k=0.1$. The computational domain is $[0,\frac{2\pi}{k}]\times[-8, 8]$. The settings are the same as in \cite{holloway2} (Table 7). For this case with a small perturbation $\alpha=0.0001$, we do not apply the Hou-Li filter to avoid numerical dissipation. 

We compute the solution up to time $T=200$ with $N_x=64$ and $N_H=256$. The time evolution of the relative deviations of discrete mass, momentum and total energy are reported in Figure \ref{fig47} (left). We can see these errors are up to machine precision for this long time simulation. We also plot the time evolution of the electric field in $L^2$ norm ($E_2$) and $L^\infty$ norm $(E_{max})$ in Figure \ref{fig47} (right). We compare them to the results by using the SLFD WENO method with mesh size $256\times512$. We can see they match each other very well. Besides both of them have a growing rate very close to $0.1084353$ from the linear theory \cite{holloway2}.

Secondly, we consider a strong perturbation with $\alpha=0.04$, $n=3$ and $k=0.1$. Other parameters are set to be $n_p=0.9$, $n_b=0.1$, $v_{d,b}=4.5$, $v_{th,p}=\sqrt{2}$, $v_{th,b}=\sqrt{2}/2$. The computational domain is still $[0,\frac{2\pi}{k}]\times[-8, 8]$. These settings have been used in \cite{shoucri1974,xiong2014high}. For this case, the Hou-Li filter is used and besides we take the scaling velocity $v_{th}=1.4$.

We take $N_x\times N_H=128\times 512$ for DGFSM and compare the solutions to SLFDM with $N_x\times N_v=160\times 320$. We first show the time evolution of the relative deviations of discrete mass, momentum and total energy in Figure \ref{fig48} (top left). These errors are also up to machine precision. We also plot the time evolution of the electric field in $L^2$ norm ($E_2$) and $L^\infty$ norm $(E_{max})$ in Figure \ref{fig48} (top right), as well as the time evolution of the total electric energy in Figure \ref{fig48} (bottom). We compare them to the results by using the SLFD WENO method with mesh size $160\times320$. We can see these results have the same structure and they are similar to those in \cite{shoucri1974}. Finally we show the surface plots of the distribution function at $T=20,50,200$ in Figure \ref{fig49}. From the comparison of these two methods, we can find that at the beginning $T=20$, the solutions are very close. But as time evolves,
the solutions are moving in different phases. However, the results from DGFSM look relatively well.

\begin{figure}
	\centering
	\includegraphics[width=3in,clip]{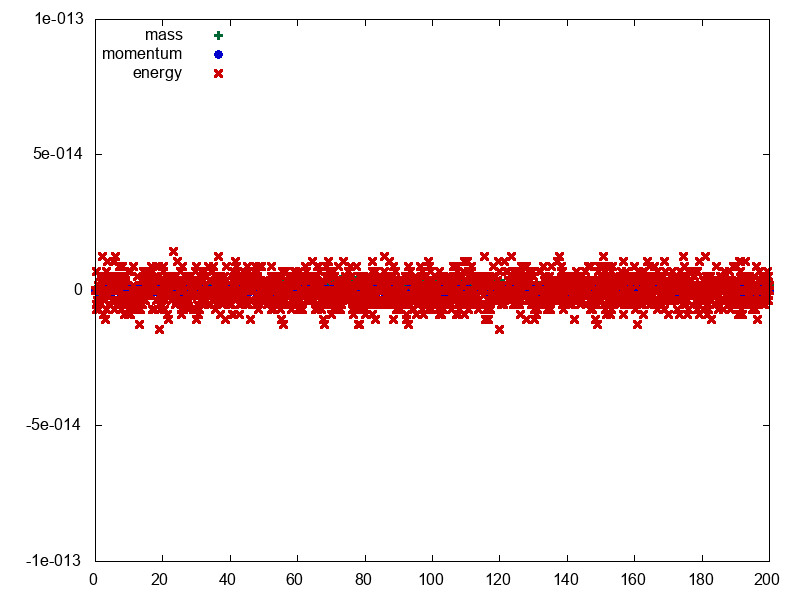}
	\includegraphics[width=3in,clip]{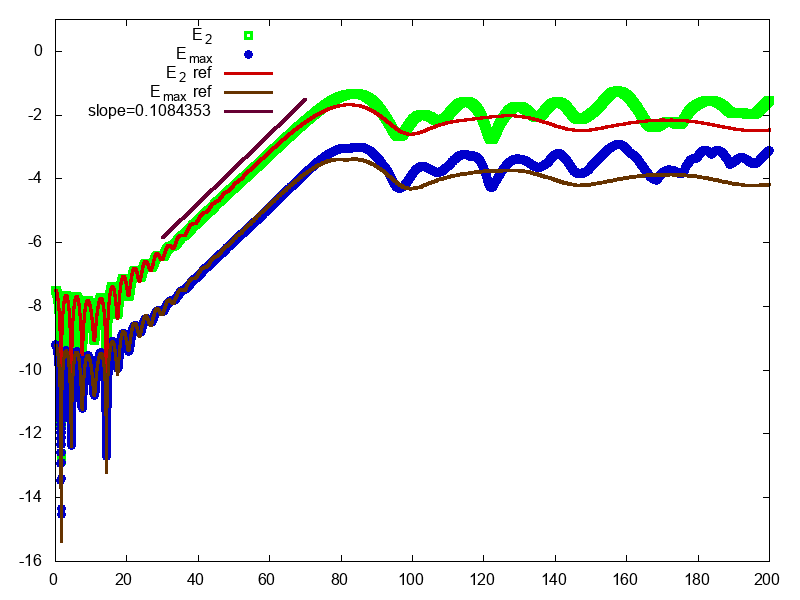}\\
	\caption{Bump-on-tail instability with the initial condition (\ref{bot}) for a small perturbation $\alpha=0.0001$. $k=0.1$ and $n=10$. Left: deviation of mass, momentum and energy. Right: time evolution of the electric field in $L^2$ norm ($E_2$) and $L^\infty$ norm ($E_{max}$) (logarithmic value). DGFSM: $N_x\times N_H = 64 \times 256$. The reference solution is from SLFDM with $N_x\times N_v=160\times 1024$.}
	\label{fig47}
\end{figure}

\begin{figure}
	\centering
	\includegraphics[width=3in,clip]{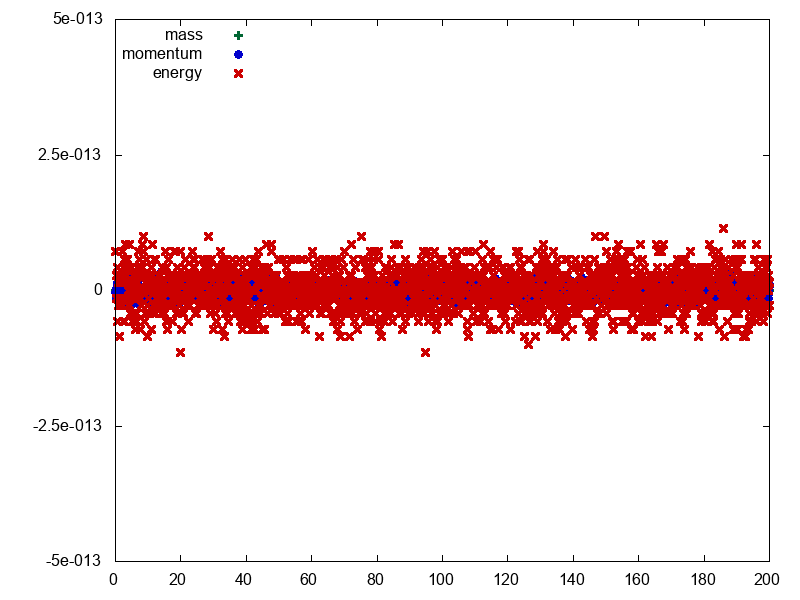}
	\includegraphics[width=3in,clip]{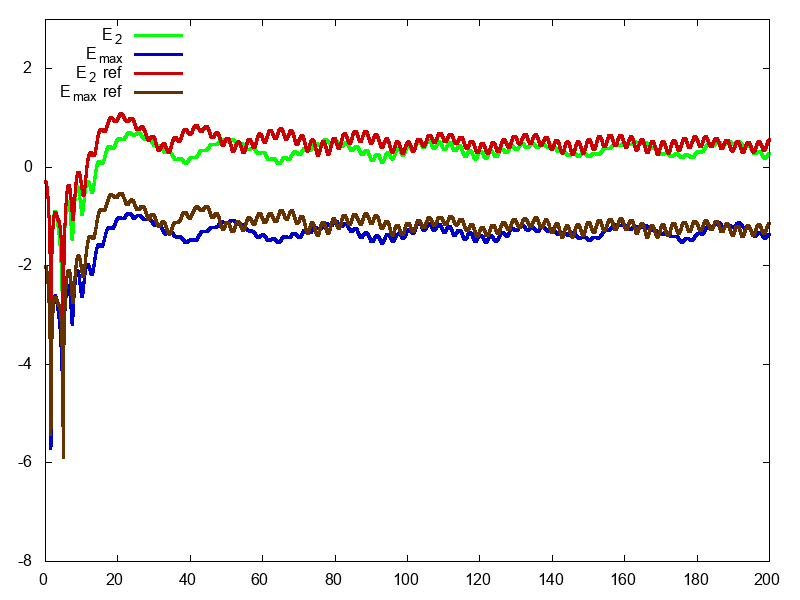} \\
	\includegraphics[width=3in,clip]{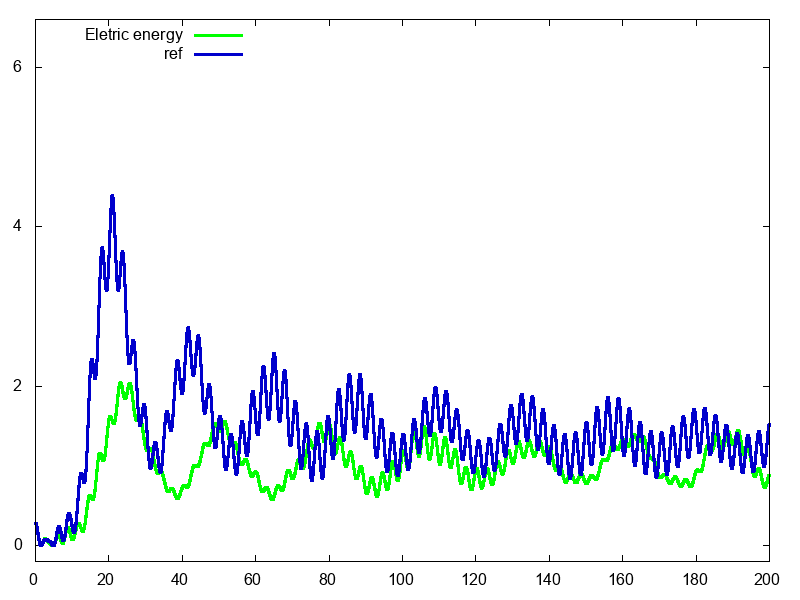},
	\caption{Bump-on-tail instability with the initial condition (\ref{bot}) for a strong perturbation $\alpha=0.04$. $k=0.1$ and $n=3$. Top left: deviation of mass, momentum and energy. Top right: time evolution of the electric field in $L^2$ norm ($E_2$) and $L^\infty$ norm ($E_{max}$) (logarithmic value). Bottom: time evolution of the electric energy. DGFSM: $N_x\times N_H = 128 \times 512$. The reference solution is from SLFDM with $N_x\times N_v=160\times 320$.}
	\label{fig48}
\end{figure}

\begin{figure}
	\centering
	\includegraphics[width=3.2in,clip]{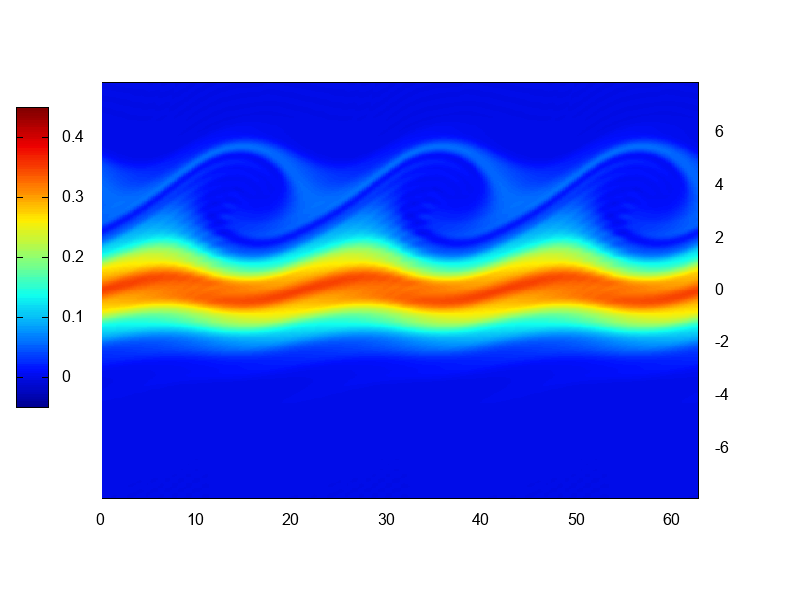},
	\includegraphics[width=3.2in,clip]{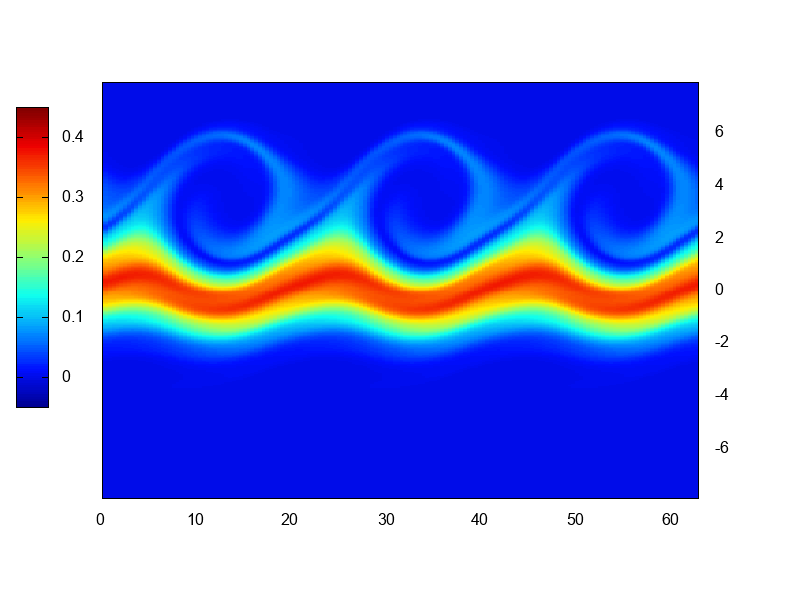}
	\includegraphics[width=3.2in,clip]{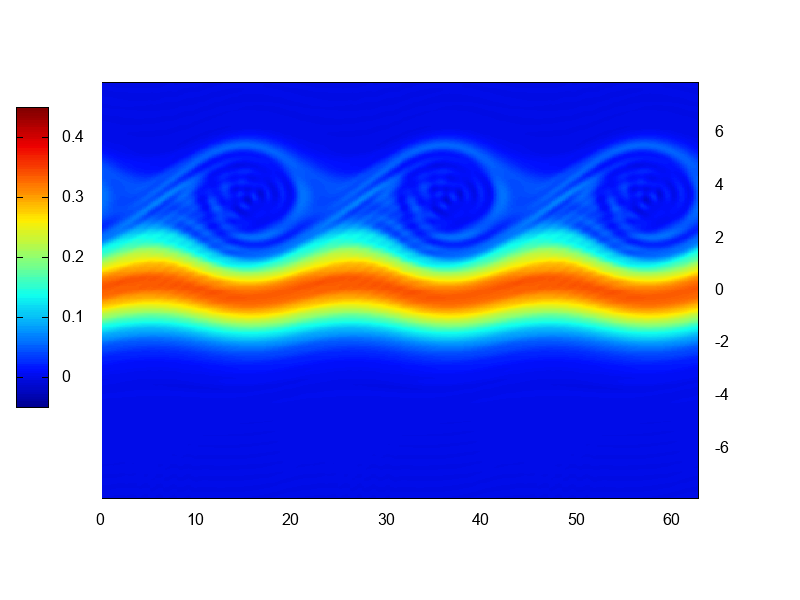}
	\includegraphics[width=3.2in,clip]{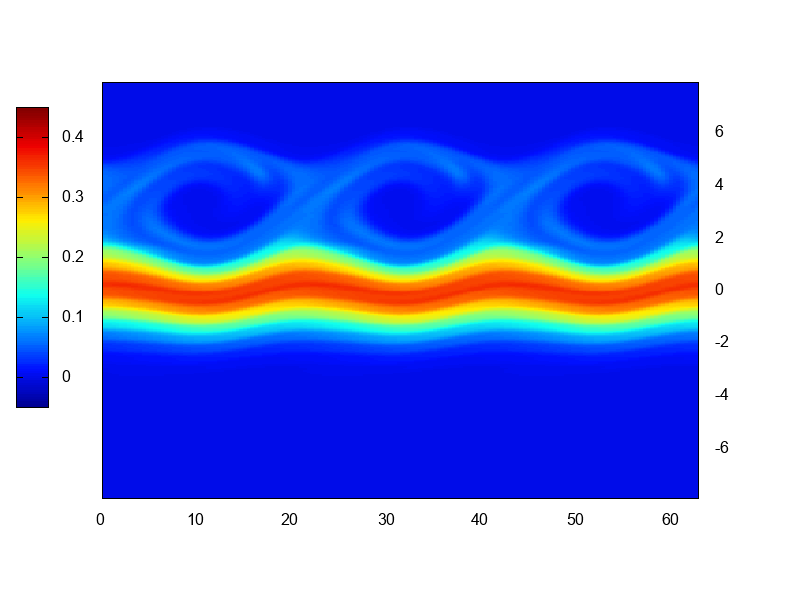}
	\includegraphics[width=3.2in,clip]{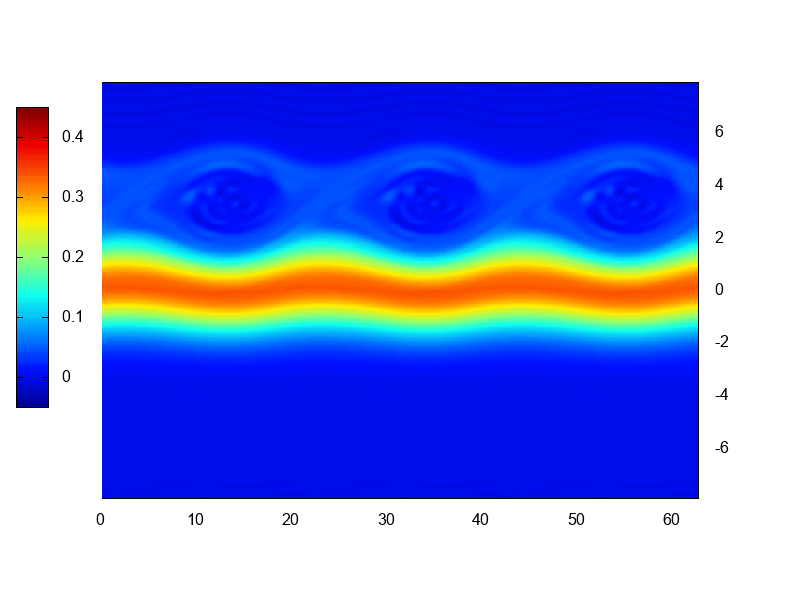}
	\includegraphics[width=3.2in,clip]{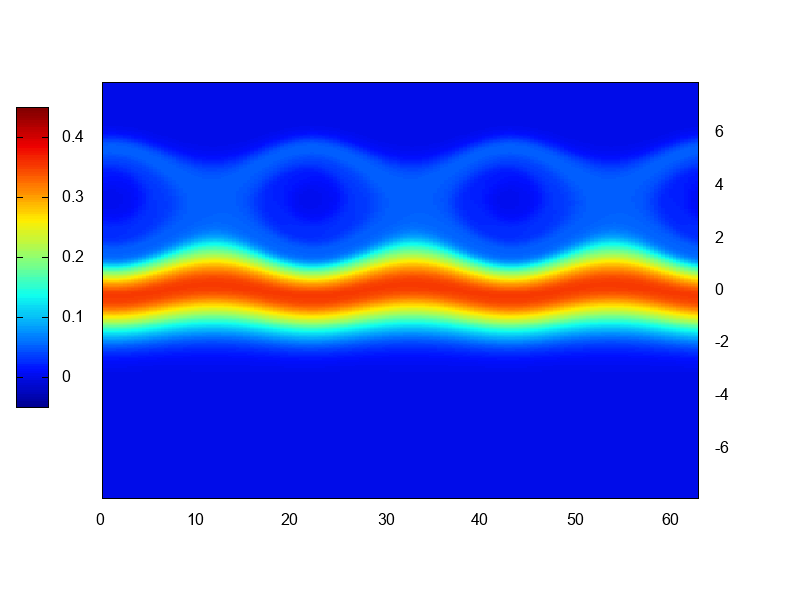}
	\caption{Bump-on-tail instability with the initial condition (\ref{bot}) for a strong perturbation $\alpha=0.04$. $k=0.1$ and $n=3$. Surface plots of the distribution function $f$ at $T=20, 50, 200$ (from top to bottom). Left: $N_x \times N_H = 128 \times 512$ for DGFSM; right: $N_x\times N_v=160\times 320$ for SLFDM. }
	\label{fig49}
\end{figure}

%%%%%%%%%%%%%%%%%%%%%%%%%%%%%%%%%%
%
%%%%%%%%%%%%%%%%%%%%%%%%%%%%%%%%%%
\section{Conclusion and perspectives}
\label{sec:5}
\setcounter{equation}{0}
\setcounter{figure}{0}
\setcounter{table}{0}

In this paper, we introduced a new method to solve the Vlasov
equation using Hermite polynomials for the velocity variable and
discontinuous Galerkin methods for space discretization. This method enforces the
conservation of  global mass, momentum and energy. A filtering
algorithm  allows to control spurious oscillations without affecting
the conservation properties. Therefore, our approach  allows to
treat strong nonlinear problems without generating numerical
instabilities. Numerical results show that our scheme is as
accurate as semi-lagrangian methods and the local reconstruction is well suited to do
parallel computation in high dimension. 

The main advantage of the present approach, using Hermite functions for
the velocity variable, is that it provides the evolution of
macroscopic quantities, that is, density, momentum, energy and higher order
moment in velocity of the distribution function. Of course, it would be
interesting to adapt the number of moments during the time evolution according
to the variations of the distribution function as in \cite{manzini2016}. In a
further study we would like to investigate the extension to the
Vlasov-Maxwell system and the  coupling of this
algorithm with fluid solvers in order to solve the kinetic equation
only where it is useful \cite{filbet-rey,filbet-xiong}.

%%%%%%%%%%%%%%%%%%%%%%%%%%%%%%%%%%%
%
%%%%%%%%%%%%%%%%%%%%%%%%%%%%%%%%%%%
\section*{Acknowledgement}

Francis Filbet is supported by the EUROfusion Consortium and has
received funding from the Euratom research and training programme
2014-2018 under grant agreement No 633053. The views and opinions
expressed herein do not necessarily reflect those of the European Commission.

T. Xiong acknowledges support by NSFC grant No. 11971025, NSF grant of Fujian Province No. 2019J06002 and NSAF No. U1630247.

{This work started when the first author visited the Tianyuan Mathematical Center in Southeast China at Xiamen University in June 2019 for a research group discussion. The author would like to thank the Tianyuan Mathematical Center for its hospitality.}

%%%%%%%%%%%%%%%%%%%%%%%%%%%%%%%%%%%
%
%%%%%%%%%%%%%%%%%%%%%%%%%%%%%%%%%%%
\section*{Conflict of interest}
On behalf of all authors, the corresponding author states that there is no conflict of interest.

%%%%%%%%%%%%%%%%%%%%%%%%%%%%%%%%%%%%%%%%%%
%
%%%%%%%%%%%%%%%%%%%%%%%%%%%%%%%%%%%%%%%%%%

\bibliographystyle{abbrv}
\bibliography{refer}
\end{document}